\documentclass[12pt]{amsart}
\usepackage{amscd,amsmath,amsthm,amssymb}
\usepackage[left]{lineno}
\usepackage{color}
\usepackage{stmaryrd}
\usepackage[utf8]{inputenc}
\usepackage{cleveref}
\usepackage{epstopdf}
\usepackage{graphicx}
\usepackage{xcolor}
\usepackage[misc,geometry]{ifsym}

\usepackage{tikz}
\definecolor{verylight}{gray}{0.97}
\definecolor{light}{gray}{0.9}
\definecolor{medium}{gray}{0.85}
\definecolor{dark}{gray}{0.6}

 \def\KK{{\NZQ K}}
 
 \def\G{{\mathcal G}}

 \def\0b{{\mathbf 0}}

\def\reg{{\mathbf reg}}

\def\KK{{\mathbb K}}
\def\height{\operatorname{ht}}
\def\depth{\operatorname{depth}}
 \def\opn#1#2{\def#1{\operatorname{#2}}} % to make operators
 \opn\chara{char} \opn\length{\ell} \opn\pd{pd} \opn\rk{rk}
 \opn\projdim{proj\,dim} \opn\injdim{inj\,dim} \opn\rank{rank}
 \opn\depth{depth} \opn\grade{grade} \opn\height{height}
 \opn\embdim{emb\,dim} \opn\codim{codim}
 
 \opn\Tr{Tr} \opn\bigrank{big\,rank}
 \opn\superheight{superheight}\opn\lcm{lcm}
 \opn\trdeg{tr\,deg}%\emph{
 \opn\reg{reg} \opn\lreg{lreg} \opn\ini{in} \opn\lpd{lpd}
 \opn\size{size} \opn\sdepth{sdepth}
 \opn\link{link}\opn\fdepth{fdepth}\opn\lex{lex}
 \opn\tr{tr}
 \opn\type{type}
 \opn\gap{gap}
 \opn\arithdeg{arith-deg}
 \opn\HS{HS}
 \opn\GL{GL}
 \opn\div{div} \opn\Div{Div} \opn\cl{cl} \opn\Cl{Cl}
 \opn\Spec{Spec} \opn\Supp{Supp} \opn\supp{supp} \opn\Sing{Sing}
 \opn\Ass{Ass} \opn\Min{Min}\opn\Mon{Mon}
 \opn\Ann{Ann} \opn\Rad{Rad} \opn\Soc{Soc}\opn\Deg{Deg}
 \opn\Im{Im} \opn\Ker{Ker} \opn\Coker{Coker} \opn\Am{Am}
 \opn\Hom{Hom} \opn\Tor{Tor} \opn\Ext{Ext} \opn\End{End}
 \opn\Aut{Aut} \opn\id{id}
 
 \opn\nat{nat}
 \opn\pff{pf}%   \pf exists already
 \opn\Pf{Pf} \opn\GL{GL} \opn\SL{SL} \opn\mod{mod} \opn\ord{ord}
 \opn\Gin{Gin} \opn\Hilb{Hilb}\opn\sort{sort}
 \opn\PF{PF}\opn\Ap{Ap}
 \opn\mult{mult}
 \opn\bight{bight}
 \opn\match{match}
\opn\St{St}

 \opn\aff{aff}
 \opn\relint{relint} \opn\st{st}
 \opn\lk{lk} \opn\cn{cn} \opn\core{core} \opn\vol{vol}  \opn\inp{inp} \opn\nilpot{nilpot}
 \opn\link{link} \opn\star{star}\opn\lex{lex}\opn\set{set}
 \opn\width{wd}
 \opn\Fr{F}
 \opn\QF{QF}
 \opn\G{G}
 \opn\type{type}\opn\res{res}
 \opn\conv{conv}
 \opn\Ind{Ind}
 \opn\gr{gr}
 
 \def\pot#1#2{#1[\kern-0.28ex[#2]\kern-0.28ex]}

 \opn\dirlim{\underrightarrow{\lim}}
 \opn\inivlim{\underleftarrow{\lim}}
 \let\to=\rightarrow
 
 \def\Implies{\ifmmode\Longrightarrow \else
         \unskip${}\Longrightarrow{}$\ignorespaces\fi}
 \def\implies{\ifmmode\Rightarrow \else
         \unskip${}\Rightarrow{}$\ignorespaces\fi}
 \def\iff{\ifmmode\Longleftrightarrow \else
         \unskip${}\Longleftrightarrow{}$\ignorespaces\fi}

 \let\:=\colon
 \newtheorem{Theorem}{Theorem}[section]
 \newtheorem{Lemma}[Theorem]{Lemma}
 
 \newtheorem{Proposition}[Theorem]{Proposition}
 \newtheorem{Remark}[Theorem]{Remark}
 
 \newtheorem{Example}[Theorem]{Example}
 
 \newtheorem{Definition}[Theorem]{Definition}

 \newtheorem{Setting}[Theorem]{Setting}

 \let\epsilon\varepsilon
 \let\kappa=\varkappa
 \def\qed{\ifhmode\textqed\fi
       \ifmmode\ifinner\quad\qedsymbol\else\dispqed\fi\fi}
 \def\textqed{\unskip\nobreak\penalty50
        \hskip2em\hbox{}\nobreak\hfil\qedsymbol
        \parfillskip=0pt \finalhyphendemerits=0}
 \def\dispqed{\rlap{\qquad\qedsymbol}}
 
 \opn\dis{dis}
 \def\pnt{{\raise0.5mm\hbox{\large\bf.}}}
 
 \opn\Lex{Lex}

\begin{document}
%\date{\today}

\title{The $\circ$ operation and $*$   operation of  fan graphs}

\author{Guangjun Zhu$^{\ast}$, Yijun Cui,  Yulong Yang and Yi yang}

%\author{Guangjun Zhu$^{1}$, Yijun Cui$^{1}$,  Yulong Yang$^{1}$ and Yi yang$^{1}$}

%\date{\today }

\address{ School of Mathematical Sciences, Soochow University, Suzhou, Jiangsu, 215006, P. R. China}

\email{zhuguangjun@suda.edu.cn(Corresponding author:Guangjun Zhu),
\linebreak[4]	237546805@qq.com(Yijun Cui), 1975992862@qq.com(Yulong Yang), %\linebreak[4]
3201088194@qq.com(Yi Yang).}

	\thanks{$^{\ast}$ Corresponding author}

\maketitle
\begin{abstract}
Let $G$  be a  finite simple graph  on the vertex set $V$ and let $I_G$ denote its edge ideal in the  polynomial ring $S=\KK[x_V]$. In this paper,  we compute the depth and the Castelnuovo--Mumford regularity of
$S/I_G$ when $G=F_{k}^{W}(K_n)$ is a $k$-fan graph, or  $G=G_1\circ G_2$ or $G=G_1* G_2$ is the graph obtained from  fan graphs $G_1$, $G_2$ by $\circ$ operation or $*$ operation, respectively.
\\
\\
{\bf Keywords}:  Regularity,  depth, fan graph, $\circ$ operation, $*$ operation\\
\\
{\bf Mathematics Subject Classification}.
Primary 13C15, 13A15, 13D02; Secondary 05E40
\end{abstract}

\section*{Introduction}

 For any two monomial ideals $I$ and $J$, it is known that $\reg(I+J)\le \reg(I)+\reg(J)-1$, see \cite{KM} and \cite{He}. Suppose we are
given a finite simple graph $G=(V(G),E(G))$ with the vertex set $V(G)$ and the edge set $E(G)$ and its  subgraphs $G_1=(V(G_1),E(G_1))$ and $G_2=(V(G_2),E(G_2))$ such that  $E(G)=E(G_1)\cup E(G_2)$, then for the edge ideal  the above inequalities imply
$\reg(I_G)\le \reg(I_{G_1})+\reg(I_{G_2})-1$.
 Herzog et al. in \cite{HMR} introduced   the concept of splitting graphs of a graph.  For two finite simple graphs  $G=(V(G),E(G))$ and $G'=(V(G'),E(G'))$,
% For a finite simple graph  $G=(V(G),E(G))$ with the vertex set $V(G)$ and the edge set $E(G)$.
 if there exists a surjective map $\alpha: V(G')\to V (G)$ such
that $\alpha(e):=\{\alpha(u), \alpha(v)\}$ is an edge of $G$ for all edges $e=\{u,v\}$ of $G'$, and such
that the map $E(G')\to E(G)$, $e\mapsto \alpha(e)$ is bijective. We call  $G'$ a splitting graph of $G$ and  $\alpha$ a splitting map of $G$. They proved in \cite[Theorem 1.3 and Proposition 1.6]{HMR} that $\reg(I_G)\le \reg(I_{G'})$ if $\reg(I_G)=\vartheta(G)+1$  where $\vartheta(G)$ is the induced matching number of $G$, or the  splitting map $\alpha: V(G')\to V (G)$ is special.
They also proved that  for path graphs and cycle graphs of even length, $\depth(S_{G'}/I_{G'})\ge \depth(S_G/I_G)$, where $S_{G}$ and $S_{G'}$ are the
polynomial rings over a   field $\KK$ in the variables corresponding to $V(G)$ and $V(G')$, respectively.

In this paper we are interested in whether formulas for the regularity and depth of $S_G/I_G$ can  be accurately determined  in terms of this information of $S_{G_1}/I_{G_1}$ and $S_{G_2}/I_{G_2}$ for a splitting graph $G':=G_1\sqcup G_2$ of graph $G$ and  subgraphs of $G'$. So far, there are no results, and it turns out to be a very  hard problem.

% For  classifing the Cohen-Macaulayness of the binomial edge ideals of bipartite graphs,
Bolognini et al. introduced in \cite{BMS} a class of chordal graphs  $F_k^{W}(K_n)$, called  $k$-fan graphs, and  a family of simple graphs  obtained from some graphs
 by the $*$ operation or the $\circ$ operation (see Definition \ref{circ_*_operations}). In Section $2$, we  give some formulas for the dimension, depth and regularity of  the quotient ring  of the edge ideal of a $k$-fan graph $F_k^{W}(K_n)$.  If  $G=(G_1,f_1)*(G_2,f_2)$ where $f_i$ is a  leaf of $G_i$ for $i=1,2$. We set $G'=G_1\sqcup G_2$ and  define a map $\alpha: V(G')\to V (G)$ by  $\alpha(f_1)=\alpha(f_2)=f$ and $\alpha(u)=u$ with $u\in V(G')\backslash \{f_1,f_2\}$. It is clear that $G'$ is a splitting graph of $G$. If   $G=(G_1,f_1)\circ (G_2,f_2)$ where $f_i$ is a  leaf of $G_i$ and $v_i$ is its  neighbor vertex in  $G_i$  for $i=1,2$. We set $G'=(G_1\backslash f_1) \sqcup (G_2\backslash f_1)$ and  define a map $\alpha': V(G')\to V (G)$ by  $\alpha'(v_1)=\alpha'(v_2)=v$ and $\alpha'(u)=u$ with $u\in V(G')\backslash \{v_1,v_2\}$. It is also clear that $G'$ is a splitting graph of $G$ and $G'$ is also an induced subgraph of $G_1\sqcup G_2$.
 In Section $3$,  we give some formulas for the depth and regularity of  some graphs  obtained from the fan graphs
 by  the two special gluing operations mentioned earlier.

\section{Preliminary}
In this section, we gather together the needed definitions and basic facts, which will
be used throughout this paper. However, for more details, we refer the reader to \cite{BM,HH,V}.

\subsection{Notions of simple graphs}
For any finite simple graph $G=(V(G),E(G))$, with a set of vertices $V(G)$ and a set of edges $E(G)$, we define some graph-theoretic notions as follows.

For a vertex $v\in V(G)$, its  \emph{neighborhood}  is defined as  $N_G(v)=\{u\,|\, \{u,v\}\in E(G)\}$  and its degree, denoted by $\deg_G(v)$, is $|N_G(v)|$. Set $N_G[v]=N_G(v)\cup\{v\}$.
For $A\subset V(G)$, $G|_A$ denotes the \emph{induced subgraph} of $G$ on the  set $A$, i.e., for $i,j \in A$, $\{i,j\} \in E(G|_A)$ if and only if $\{i,j\}\in E(G)$. For $W\subseteq V(G)$, we denote by $G\backslash W$ the induced subgraph of  $G$  on $V(G) \setminus W$. For a vertex $v\in V(G)$,  we denote by $G\backslash v$  the induced subgraph of $G$ on the  set $V(G)\backslash \{v\}$ for simplicity.

A subset $M\subset E(G)$ is a \emph{matching} of $G$ if $e\cap e'=\emptyset$  for any pair of edges $e, e'\in M$.  A matching $M$ of $G$ is called an \emph{induced matching} if the induced subgraph at
 the vertices of $M$ contains no edges other than those already contained in M.  The \emph{induced matching number} of $G$, denoted by $\vartheta(G)$, is the maximum size of an induced matching in $G$.

A \emph{complete graph} (or clique) on $n$ vertices is a graph where for any two vertices there is an edge connecting  them. It is denoted by $K_n$. A \emph{cycle} of length $n$ in a graph  is a closed walk along its edges, $\{x_1, x_2\}, \{x_2, x_3\},\ldots, \{x_{n-1},x_n\}, \{x_n, x_1\}$,
such that $x_i\ne x_j$ for $i\ne j$. We denote the cycle on $n$ vertices by $C_n$. A \emph{chord} in the cycle $C_n$ is an edge $\{x_i, x_j\}$ where $x_j\ne x_{i-1}, x_{i+1}$. A graph is said to be \emph{chordal} if it
 has a chord for every cycle of length greater than or equal to  $4$. It is clear that a complete graph is a chordal graph.
for every cycle of length greater than or equal to $4$ there is a chord. It is clear that a complete graph is a chordal graph.

For a simple graph $G=(V(G),E(G))$,  the subset $C\subseteq V(G)$ is called a \emph{vertex cover}  of $G$ if $\{x_i,x_j\}\in E(G)$  then $x_i\in C $ or $x_j\in C$. A vertex cover $C$ of $G$ is called \emph{minimal} if every proper subset of $C$ is not a vertex cover of $G$.

\subsection{Algebraic preliminaries}

Let  $S=\KK[x_1,\ldots,x_n]$ be a polynomial ring   over  a field $\KK$. Let  $M$ be a graded $S$-module with minimal free resolution
$$0\rightarrow \bigoplus\limits_{j}S(-j)^{\beta_{p,j}}\rightarrow \bigoplus\limits_{j}S(-j)^{\beta_{p-1,j}}\rightarrow \cdots\rightarrow \bigoplus\limits_{j}S(-j)^{\beta_{0,j}}\rightarrow M\rightarrow 0,$$
where the maps are exact, $p\le n$, and $S(-j)$ is the free module obtained by  shifting the degrees in $S$ by $j$.
The numbers $\beta_{i,j} $'s are positive integers and are called   the $(i, j)$th graded Betti number of $M$.
Two very important homological invariants   related to these numbers are  the Castelnuovo-Mumford regularity (or simply regularity) and  the depth, denoted by $\reg(M)$ and $\depth(M)$ respectively:
%The Castelnuovo-Mumford regularity (or simply regularity),  denoted by $\reg(M)$, is defined to be
	\begin{align*}
\reg(M)&=\mbox{max}\,\{j-i\ |\ \Tor_i(M,\KK)_j\neq 0\}\\
\depth(M)&=\mbox{min}\,\{i\ |\ \Ext^{i}(\KK,M)\neq 0\}.
%\pd(M)&=\mbox{max}\,\{i\ |\ \text{there exists some }j, \beta_{i,j}(M)\neq 0\}.
	\end{align*}

The following lemmas are often used  to compute  the depth and regularity of a module or  ideal.
In particular, since the facts in Lemma \ref{lem:direct_sum} are well-known, they are used implicitly in this paper.

\begin{Lemma}
	\label{lem:direct_sum}
	Let $M,N$ be two  finitely generated graded $S$-modules. Then,
	\begin{itemize}
		\item[(1)] %\label{direct sum-1}
		$\depth(M\oplus N)=\min\{\depth(M),\depth(N)\}$, and
		\item[(2)] %\label{direct sum-b}
		$\reg(M\oplus N)=\max\{\reg(M),\reg(N)\}$.
	\end{itemize}
\end{Lemma}

\begin{Lemma} {\em (\cite[Lemmas 2.1 and 3.1]{HT})}
	\label{exact}
	Let $0\longrightarrow M\longrightarrow N\longrightarrow P\longrightarrow 0$  be an exact sequence of finitely generated graded $S$-modules. Then we have
	\begin{enumerate}
		%\item 	\label{exact-1} $\depth\,(N)\geq \min\{\depth\,(M), \depth\,(P)\}$, the equality holds if $\depth\,(P) \neq \depth\,(M)-1$.
		\item 	\label{exact-2}$\depth\,(M)\geq \min\{\depth\,(N), \depth\,(P)+1\}$, the equality holds if $\depth\,(N)\\ \neq \depth\,(P)$.
		%	\item[(6)]$\depth\,(P)\geq \min\{\depth\,(M)-1, \depth\, (N)\}$, the equality holds if $\depth\,(M)\\ \neq \depth\,(N)$.	
		%\item 	\label{exact-3} $\reg\,(N)\leq \max\{\reg\,(M), \reg\,(P)\}$, the equality holds if $\reg\,(P) \neq \reg\,(M)-1$.
		\item 	\label{exact-4} $\reg\,(M)\leq \max\{\reg\,(N), \reg\,(P)+1\}$, the equality holds if $\reg\,(N) \neq \reg\,(P)$.
		%	\item[(3)]$\reg\,(P)\leq \max\{\reg\,(M)-1, \reg\,(N)\}$, the  equality holds if $\reg\,(M) \neq \reg\,(N)$.
	\end{enumerate}
\end{Lemma}

\begin{Lemma}{\em (\cite[Lemma 2.2, Lemma 3.2]{HT})}
	\label{sum}
	Let $S_{1}=\KK[x_{1},\dots,x_{m}]$ and $S_{2}=\KK[x_{m+1},\dots,x_{n}]$ be two polynomial rings  over $\KK$,  $I\subset S_{1}$ and
	$J\subset S_{2}$ be two non-zero homogeneous  ideals. Let $S=S_1\otimes_\KK S_2$.  Then we have
	\begin{enumerate}
		\item \label{sum-1} $\reg\,(S/(I+J))=\reg\,(S_1/I)+\reg\,(S_2/J)$;
		\item \label{sum-2} $\depth\,(S/(I+J))=\depth\,(S_1/I)+\depth\,(S_2/J)$;
		%\item \label{sum-3} $\reg\,(S/JI)=\reg\,(S_1/I)+\reg\,(S_2/J)+1$;
		%\item \label{sum-4}$\depth\,(S/JI)=\depth\,(S_1/I)+\depth\,(S_2/J)+1$.
	\end{enumerate}
	%In particular,  if  $u$ is a monomial of degree  $d$ such that $\mbox{supp}\,(u)\cap \mbox{supp}\,(I)=\emptyset$, let $J=(u)$,  then  $\mbox{reg}\,(J)=d$ and $\mbox{reg}\,(JI)=\mbox{reg}\,(I)+d$.
\end{Lemma}

\begin{Lemma}{\em (\cite[Lemma 1.3]{HTT})}
	\label{quotient}
	%Let $u\in S$ be a monomial of degree $d$ and $J=(u)$, and let $I\subset S$ be a proper non-zero homogeneous ideal. Then
	Let  $I\subset S$ be a proper non-zero homogeneous ideal. Then
	\[
	\reg\,(I)=\reg\,(S/I)+1.
	\]
	%\begin{itemize}
	%		\item[(1)] $\reg\,(I)=\reg\,(S/I)+1$;
	%\item[(2)] $\depth\,(I)=\depth\,(S/I)+1$;
	%		\item[(2)] $\depth\,(S/J)=d-1$.
	%	\end{itemize}
\end{Lemma}

For a simple graph  $G=(V(G),E(G))$,  its edge ideal is defined as follows:
\[
I_G=(x_ix_j \mid  \{x_i, x_j\}\in E(G))\subset S],
\]
where $S=\KK[x_{V(G)}]$ is the polynomial ring over a field $\KK$ and   $x_{V(G)}=\{x_i|i\in V(G)\}$.

For a monomial ideal $I$, let $\mathcal{G}(I)$ be its  unique  minimal set of monomial generators. For a subset $A\subset [n]$, we set $x_{A}:=\{x_i|i\in A\}$ for simplicity. The following lemma is very important for the whole paper.
\begin{Lemma}
	\label{decomposition}
	Let $G=(V,E)$ be  a simple connected  graph. Let $J=(x_{N_G(v)})+I_{G \backslash N_G[v]}$ and $K=(x_v)+I_{G \backslash v}$, where $v\in V$. Then,
	\begin{enumerate}
		\item  \label{decomposition-1} $J+K=(x_{N_G[v]})+I_{G \backslash N_G[v]}$;
		\item \label{decomposition-2} $I_G=J\cap K$;
		\item \label{decomposition-3} $\depth(S/J)=\depth(S/(J+K))+1$;
       \item \label{decomposition-4} $\reg(S/J)=\reg(S/(J+K))$.
\end{enumerate}
\end{Lemma}
\begin{proof} (1) Since  $\mathcal{G}(I_{G \backslash v})\subseteq  (x_{N_G(v)})+\mathcal{G}(I_{G \backslash N_G[v]})$, we have
\begin{align*}
	J+K=(x_{N_G(v)})+I_{G \backslash N_G[v]}+(x_v)+I_{G \backslash v}
	=(x_{N_G[v]})+I_{G \backslash N_G[v]}.
\end{align*}

(2) It is clear that $I_G\subseteq J\cap K$. On the other hand, since $(x_{N_G(v)})\cap (x_v)\subseteq I_G$,  both $G\backslash v$	and $G\backslash N_G[v]$ are induced subgraphs of $G$. We have
	\begin{align*}	J\cap K&=[(x_{N_G(v)})+I_{G \backslash N_G[v]}]\cap [(x_v)+I_{G \backslash v}]\\
	&=(x_{N_G(v)})\cap (x_v)+(x_{N_G(v)})\cap I_{G \backslash v}+I_{G \backslash N_G[v]}\cap (x_v)
	+I_{G \backslash N_G[v]}\cap I_{G \backslash v} \\
	& \subseteq I_G.
\end{align*}

(3) and (4) follow from (1). % \qedhere
\end{proof}

\begin{Lemma}{\em (\cite[Corollary 4.6]{CZW})}
	\label{complete}
 Let $G$ be a complete graph on  the set $[n]$, then
 \[
 \depth\,(S/I_G)=\reg\,(S/I_G)=1.
 \]
\end{Lemma}

\begin{Lemma}{\em (\cite[Theorem 3.3, Corollary 3.3]{Zhu1})}
	\label{path}
 Let $n\ge 2$ be an integer and $P_n$ be  a path  graph on the  set $[n]$, then
 \[
 \depth(S/I_G)= \lceil\frac{n}{3}\rceil  \text{\ and\ }  \reg(S/I_G)=\lfloor\frac{n+1}{3}\rfloor
 \]
where  $\lceil\frac{n}{3}\rceil$ is the smallest integer $\ge \frac{n}{3}$ and $\lfloor\frac{n+1}{3}\rfloor$ is the  largest integer $\le \frac{n+1}{3}$.
\end{Lemma}

 %A graph $G$ is chordal if every induced cycle in $G$ has length $3$.
 \begin{Lemma}{\em (\cite[Corollary 6.9]{HT1})}
\label{chordal}
		If $G$ is a chordal graph, then $\reg(S_G/I_{G})=\vartheta(G)$, where $\vartheta(G)$ is the induced matching number of $G$.
	\end{Lemma}

\section{Study of $F_k^W(K_n)$}
 Bolognini et al. in \cite{BMS} introduced the fan graphs of complete graphs, which are a family of chordal graphs.
In this section, we will study the dimension, depth and regularity of  the quotient ring  of the edge ideal of these graphs. We start with the definition reformulated in \cite{JK} and \cite{SZ}.

\begin{Definition}
    [{\cite[Definition 3.1]{SZ}}]
    \label{fan}
    Let $K_n$ be a complete graph on the   set $[n]$.
   \begin{enumerate}%[a]
   	\item Let $U=\{u_1,u_2,\ldots,u_r\}$ be a subset of $[n]$. Suppose that for each $i\in [r]$, a new complete graph $K_{a_i}$ with $a_i>i$ is attached to $K_n$ in such a way that $V(K_n)\cap V(K_{a_i})= \{u_1,u_2, \dots,u_i\}$. We say that the resulting graph is obtained by \emph{adding a fan} to $K_n$ on the set $U$ and $\{K_{a_1},\ldots,K_{a_r}\}$ is the \emph{branch} of that fan on $U$.
   	
   	\item Let $W$ be a subset of $[n]$ and suppose that $W=W_1\sqcup\cdots\sqcup W_k$ is a partition of $W$. Let $F^W_k(K_n)$ be a graph obtained from $K_n$ by adding a fan on each  $W_i$. The resulting graph $F^W_k(K_n)$ is called a \emph{$k$-fan graph} of $K_n$ on the set $W$. For future reference, for each $i\in [k]$, we assume that $\{K_{a_{i,1}},\ldots,K_{a_{i,r_i}}\}$ is the branch of the fan on $W_i=\{w_{i,1},\ldots,w_{i,r_i}\}$. For notational convenience, we also set $h_{i,j}:= a_{i,j}-j$.
   	
   	\item A branch $\{K_{a_{i,1}},\ldots,K_{a_{i,r_i}}\}$ of the fan on $W_i$ is  called a \emph{pure branch} if $h_{i,j}=1$ for every $j\in [r_i]$. Furthermore, if each branch is pure, then $F^W_k(K_n)$ is said to be a \emph{$k$-pure fan graph} of $K_n$ on $W$.
   	
   	\item The complete graph $K_n$ can be considered as a degenerate fan without branches, i.e., with $k=0$.
   \end{enumerate}
\end{Definition}

\begin{Example} \label{example1} Below are examples of a $2$-fan graph and  a $1$-fan graph.
	\vspace{0.5cm}
	\begin{center}
		\begin{tikzpicture}[thick,>=stealth]
			%\draw[help lines] (-10,-10) grid (10,10);
			\setlength{\unitlength}{1mm}
			
			\thicklines
		
		%2-fan图
	
		\put(-60,20){\circle*{1.5}}
		\put(-70,10){\circle*{1.5}}
		\draw[fill = gray] (-7,0) circle(0.06);	
		\put(-60,-10){\circle*{1.5}}
		\put(-50,0){\circle*{1.5}}
		\put(-50,10){\circle*{1.5}}
		
		\put(-60,30){\circle*{1.5}}
		\put(-72,25){\circle*{1.5}}
		\put(-80,20){\circle*{1.5}}
		\put(-40,-15){\circle*{1.5}}
		\put(-35,5){\circle*{1.5}}
		\draw[ultra thick](-7,1)--(-6,2);
		\draw[ultra thick](-7,1)--(-5,1);
		\draw[gray](-7,1)--(-5,0);
		\draw[gray](-7,1)--(-6,-1);	
		\draw[gray](-7,1)--(-7,0);
		\draw[ultra thick](-7,1)--(-8,2);
		\draw[ultra thick](-7,1)--(-7.2,2.5);			
        \draw[ultra thick](-6,2)--(-7.2,2.5);	

		\draw[ultra thick](-7,1)--(-6,3);
		
		\draw[ultra thick](-6,2)--(-5,1);
		\draw[gray](-6,2)--(-5,0);
		\draw[gray](-6,2)--(-6,-1);
		\draw[gray](-6,2)--(-7,0);
		\draw[ultra thick](-6,2)--(-6,3);
		
		\draw[gray](-5,1)--(-5,0);
		\draw[gray](-5,1)--(-6,-1);
		\draw[gray](-5,1)--(-7,0);
		\draw[ultra thick](-5,1)--(-6,3);
		
		\draw[ultra thick](-5,0)--(-6,-1);
		\draw[gray](-5,0)--(-7,0);
		\draw[ultra thick](-5,0)--(-3.5,0.5);
		\draw[ultra thick](-5,0)--(-4,-1.5);
		
		\draw[gray](-6,-1)--(-7,0);
		\draw[ultra thick](-6,-1)--(-4,-1.5);
        \put(-70,-20){$2-$ $fan\, graph$};
%1-fan图
\put(-1,16){\circle*{1.8}};
\put(11,16){\circle*{1.8}};
\put(20,5){\circle*{1.8}};
\draw[fill = gray] (1.1,-0.6) circle(0.06);	
\draw[fill = gray] (-0.1,-0.6) circle(0.06);	
\draw[fill = gray] (-1,0.5) circle(0.06);	

\put(-1,30){\circle*{1.8}};
\put(11,30){\circle*{1.8}};
\put(27,25){\circle*{1.8}};
\put(37,20){\circle*{1.8}};
\put(37,11){\circle*{1.8}};
\put(-18,14){\circle*{1.8}};
\put(-13,25){\circle*{1.8}}
\draw[ultra thick](-0.1,1.6)--(1.1,1.6);
\draw[ultra thick](-0.1,1.6)--(2,0.5);
\draw[gray](-0.1,1.6)--(1.1,-0.6);
\draw[gray](-0.1,1.6)--(-0.1,-0.6);
\draw[gray](-0.1,1.6)--(-1,0.5);
\draw[ultra thick](-0.1,1.6)--(-0.1,3);
\draw[ultra thick](-0.1,1.6)--(1.1,3);
\draw[ultra thick](-0.1,1.6)--(2.7,2.5);
\draw[ultra thick](-0.1,1.6)--(3.7,2);
\draw[ultra thick](-0.1,1.6)--(3.7,1.1);
\draw[ultra thick](-0.1,1.6)--(-1.8,1.4);
\draw[ultra thick](-0.1,1.6)--(-1.3,2.5);

\draw[ultra thick](1.1,1.6)--(2,0.5);
\draw[gray](1.1,1.6)--(1.1,-0.6);
\draw[gray](1.1,1.6)--(-0.1,-0.6);
\draw[gray](1.1,1.6)--(-1,0.5);
\draw[ultra thick](1.1,1.6)--(2.7,2.5);
\draw[ultra thick](1.1,1.6)--(3.7,2);
\draw[ultra thick](1.1,1.6)--(3.7,1.1);
\draw[ultra thick](1.1,1.6)--(1.1,3);
\draw[ultra thick](1.1,1.6)--(-0.1,3);

\draw[gray](2,0.5)--(1.1,-0.6);
\draw[gray](2,0.5)--(-0.1,-0.6);
\draw[gray](2,0.5)--(-1,0.5);
\draw[ultra thick](2,0.5)--(2.7,2.5);
\draw[ultra thick](2,0.5)--(3.7,2);
\draw[ultra thick](2,0.5)--(3.7,1.1);

\draw[gray](1.1,-0.6)--(-0.1,-0.6);
\draw[gray](1.1,-0.6)--(-1,0.5);
\draw[gray](-0.1,-0.6)--(-1,0.5);
\draw[gray](-1,0.5)--(-0.1,1.6);

\draw[ultra thick](-0.1,3)--(1.1,3);

\draw[ultra thick](2.7,2.5)--(3.7,2);
\draw[ultra thick](2.7,2.5)--(3.7,1.1);

\draw[ultra thick](3.7,2)--(3.7,1.1);
		
\draw[ultra thick](-1.8,1.4)--(-1.3,2.5);
 \put(-8,-20){$1-$ $fan\,  graph$};
%\put(-30,-30){$Figure\;1$ The  examples of a $2$-fan graph and a $1$-fan graph};
\end{tikzpicture}
\end{center}
\vspace{0.8cm}
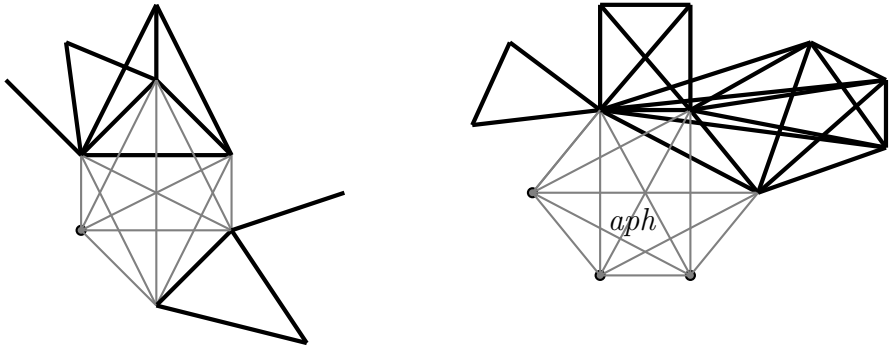
\hspace{3.0cm}Figure $1$ Example of  $2$-fan and  $1$-fan graphs
\end{Example}

For a positive integer $n\ge 2$  we set $[n]=\{1,\ldots,n\}$.  We need the following lemma.
\begin{Lemma} {\em (\cite[Lemma 9.1.4]{HH})}
	\label{prime} Let $G$ be a simple graph on  the  set $[n]$. A subset $C=\{i_1,\ldots,i_r\}\subset [n]$ is a
	vertex cover of $G$ if and only if the prime ideal $P_C=(x_{i_1},\ldots,x_{i_r})$ contains
	$I_G$. In particular, $C$ is a minimal vertex cover of $G$ if and only if $P_C$ is a
	minimal prime ideal of $I_G$.
\end{Lemma}

The following theorem provides an exact formula for the dimension  of the quotient ring of the edge ideal of a $k$-fan graph.

\begin{Theorem}
	\label{thm:fan_Dimension}
Let $G=F_k^W(K_n)$ be a $k$-fan graph of the complete graph $K_n$ on the set $W\subseteq [n]$ with $n\ge 2$  and $W=W_1\sqcup\cdots\sqcup W_k$ be a partition of $W$. Then
	\[
	\dim(S/I_G)= \begin{cases}
		|W|+1 & \text {if $W \subsetneq [n]$};\\
		n&   \text {if $W=[n]$}.
	\end{cases}
	\]
\end{Theorem}
\begin{proof}	
	Let $I_G=P_1 \cap \cdots  \cap P_m$ be   the irredundant  primary decomposition of $I_G$, where  $P_i=(x_{i_1},x_{i_2},\ldots,x_{i_{r_i}})$ for $i\in [m]$, then $\dim(S/I_G)=|V(G)|-\min \{r_i: i \in [m] \}$ and  each
	$C_i=\{i_1,i_2,\ldots,i_{r_i}\}$ is a minimal vertex cover of $G$ by Lemma \ref{prime}.
Therefore,  to compute the dimension of $S/I_G$, we only need to find the vertex cover of G with the smallest size.
It is known that for a complete graph $H$, the set  $C\subsetneq V(H)$ is its minimal vertex cover if and only if $|C|=|V(H)|-1$.
	For each $i\in [k]$, let $W_i=\{w_{i,1},\ldots,w_{i,r_i}\}$ and $\{K_{a_{i,1}},\ldots,K_{a_{i,r_i}}\}$ be the branch of the fan on $W_i$. We distinguish between the following two cases:	
	
	(i) If $W \subsetneq [n]$, 	we choose   $u_{i,j}\in V(K_{a_{i,j}})\backslash W_i$ for each $i\in [k], j\in [r_i]$ and  $v\in [n]\backslash W$. Let  $C_{i,j}=V(K_{a_{i,j}})\backslash \{u_{i,j}\}$, $C'=V(K_n)\backslash \{v\}$ and $C=(\bigcup\limits_{i=1}^{k}\bigcup\limits_{j=1}^{r_i}C_{i,j})\cup C'$. Then it is clear that  $u_{i,j},v\notin C$ and $|C|=|V(G)|-\sum\limits_{i=1}^{k}|W_i|-1=|V(G)|-|W|-1$.
	
	Claim: $C$ is a  vertex cover of $G$ of minimum size.
	
	Indeed, for any $e\in E(G)$, we have $e\in E(K_n)$ or $e\in E(K_{a_{i,j}})$ for some  $i\in[k]$, $j\in[r_i]$. It is clear that $e\cap C\ne\emptyset$ by the choice of $C$, which implies that $C$ is a vertex cover of $G$. On the other hand,  if there  exists some  vertex cover $D$ of $G$ with $|D|<|V(G)|-|W|-1$. Since there  are exactly $(|W|+1)$  maximal cliques  $K_n$ and  $K_{a_{i,j}}$ for  $i\in[k]$, $j\in[r_i]$ in $G$, there must be  an edge $\{u,v\}\in E(G)$ such that $u,v\notin D$, which contradicts with the fact that $D$ is a vertex cover of $G$.
	
	(ii) If $W=[n]$,	we choose   $u_{i,j}\in V(K_{a_{i,j}})\backslash W_i$ for every $i\in [k], j\in [r_i]$, $C_{i,j}=V(K_{a_{i,j}})\backslash \{u_{i,j}\}$  and $C=\bigcup\limits_{i=1}^{k}\bigcup\limits_{j=1}^{r_i}C_{i,j}$. So $|C|=|V(G)|-|W|$. In this case, we can obtain that $C$ is a  vertex cover of $G$ with the minimum size, using similar arguments as in (i).  So in this case we get $\dim(S/I_G)=n$ in this case. And this completes the proof.
\end{proof}	

For a monomial $u\in S$, we set $\supp(u)=\{x_i : x_i |u\}$. For  a monomial ideal $I$ with  $\mathcal{G}(I)=\{u_1,\ldots,u_m\}$, we set
$\supp(I)=\bigcup\limits_{i=1}^{m}\supp(u_i)$. We  also denote by $S_G$ the polynomial rings in the variables corresponding to $V(G)$  when more graphs are involved.

\begin{Proposition}
	\label{1-fan}
	Let $G=F_1^W(K_n)$ be  a $1$-fan graph of the complete graph  $K_n$ on the set $W \subseteq [n]$ with $n\ge 2$. Then \[\depth\,(S/I_G)=1.\]
\end{Proposition}
\begin{proof} Let $V(K_n)\cap V(K_{a_1})=\{w_1,\ldots,w_{r_1}\}$.
	We apply induction on $|W|$.  If  $|W|=1$, then $V(K_n)\cap V(K_{a_1})=\{w_1\}$. In this case, let $J=(x_{N_G(w_1)})$ and $K=(x_{w_1})+I_{K_{n-1}}+I_{K_{a_1-1}}$, then
$\depth(S/J)=1$, $\depth(S/K)=2$ by Lemma \ref{complete}, since $\{x_{w_1}\}\cap \supp(I_{K_{n-1}})\cap\supp(I_{K_{a_1-1}})=\emptyset$. Meanwhile, we have  $I_G=J\cap K$ and  $J+K=(x_{V(G)})$.
 So $\depth(S/(J+K))=0$.  Hence 	
	the desired result holds by using Lemma \ref{lem:direct_sum}(1) and Lemma\ref{exact}(\ref{exact-2}) to the following exact sequence
	\begin{equation}
		0\longrightarrow \frac{S}{J \cap K}\longrightarrow \frac{S}{J} \oplus \frac{S}{K}\longrightarrow \frac{S}{J+K} \longrightarrow 0,
		\label{eqn:SES-1}
	\end{equation}	
	Now, we assume that $|W|\ge 2$. Let $J=(x_{N_G(w_1)})$ and $K=(x_{w_1})+I_{G \backslash w_1}$, then $J+K=(x_{V(G)})$. It follows that $\depth(S/(J+K))=0$ and  $\depth(S/J)=1$. Since $G \backslash w_1$ is the disjoint union $K_{a_1-1}\sqcup F_1^{W\backslash \{w_1\}}(K_{n-1})$, % is the disjoint union of a clique $K_{a_1-1}$ and a  $1$-fan graph  $F_1^{W\backslash \{w_1\}}(K_{n-1})$$G\backslash V(K_{a_1})$ of  $K_{n-1}$ on the set $W\backslash w_1$,
we have $K=(x_{w_1})+I_{K_{a_1-1}}+I_{G\backslash V(K_{a_1})}$. Since   $|W\backslash \{w_1\}|=|W|-1$,  we know by  induction that
	\[
	\depth(S/K)=\depth(S_{K_{a_1-1}}/I_{K_{a_1-1}})+\depth(S_{G\backslash V(K_{a_1})}/I_{G\backslash V(K_{a_1})})=1+1=2.
	\]
	Again applying Lemma \ref{exact}(\ref{exact-2}) and Lemma \ref{decomposition}(\ref{decomposition-2}) to the  exact sequence (\ref{eqn:SES-1}),  we get the desired result.
\end{proof}

\begin{Theorem}
	\label{depth of F_k}
	Let $G=F_k^W(K_n)$ be a $k$-fan graph of the complete graph $K_n$ on the set $W\subseteq [n]$ with $n\ge 2$  and $W=W_1\sqcup\cdots\sqcup W_k$ be a partition of $W$. Then
	\[
	\depth(S/I_G)=1+|W| -\max\{|W_1|, |W_2|,\ldots,|W_k|\}.
	\]
\end{Theorem}

\begin{proof}
	We prove the statement by induction on $k$ and $n$. The  case for $k=1$ is shown in Proposition \ref{1-fan}. If $n=k=2$, then $|W_1|=|W_2|=1$. Let $V(K_n)\cap V(K_{a_{i,1}})=W_i$  with $W_i=\{w_{i,1}\}$ for $i=1,2$. In this case,  $G$ is a graph obtained by joining two complete graphs $K_{a_{1,1}}$ and $K_{a_{2,1}}$ with a common edge $\{w_{1,1},w_{2,1}\}$.  If  $K_{a_{1,1}}=K_{a_{2,1}}=K_2$, then in this case $G$ is a path with $4$ vertices.  So we get  $\depth(S/I_G)=2$ by Lemma \ref{path}. Now, we may assume that $K_{a_{2,1}}$ is a clique with at least three vertices. In this case, we choose $J=(x_{N_G(w_1)})+I_{G \backslash N_G[w_1]}$,  $K=(x_{w_1})+I_{G \backslash w_1}$ with $w_1=w_{1,1}$, then, by  Lemma \ref{decomposition}, we get that  $I_G=J\cap K$, $\depth\,(S/J)=\depth\,(S/(J+K))+1$ and $J+K=(x_{N_G[w_1]})+I_{G \backslash N_G[w_1]}$, where
	$G\backslash N_G[w_1]=K_{a_{2,1}-1}$ is a clique with at least two vertices. Hence $\depth(S/(J+K))=1$. Note that $G \backslash w_1$ is the disjoint union $K_{a_{1,1}-1}\sqcup K_{a_{2,1}}$,
we have $\depth(S/K)=2$. By Lemma \ref{exact}(\ref{exact-2})  and the  exact sequence (\ref{eqn:SES-1}),  we get
	$\depth(S/I_G)=2$, which confirms the claim in this case.
	
	Now we can assume that $n\ge 3$ and $k\geq 2$. Suppose that  $\max\{|W_1|,\ldots,|W_k|\}=T$ and $T=|W_1|$.
	Let $V(K_n)\cap V(K_{a_{1,r_1}})=\{w_{1,1},\ldots,w_{1,r_1}\}$.  Let $J=(x_{N_G(w_1)})+I_{G \backslash N_G[w_1]}$,  $K=(x_{w_1})+I_{G \backslash w_1}$  where $w_1=w_{1,1}$, then $I_G=J\cap K$, $\depth\,(S/J)=\depth\,(S/(J+K))+1$ and $J+K=(x_{N_G[w_1]})+I_{G \backslash N_G[w_1]}$  by   Lemma \ref{decomposition}.
	We distinguish between  the following two cases:
	
	(i) If $T=1$, then  $|W_i|=1$ for each $i\in [k]$. In this case, $G\backslash N_G[w_1]$ has $(k-1)$ connected components  consisting of some cliques and isolated vertices. Hence
	$\depth(S/(J+K))=k-1$ by Lemma \ref{complete}. At the same time, $G \backslash w_1$ is the disjoint union $K_{a_{1,1}-1}\sqcup F_{k-1}^{W\backslash \{w_1\}}(K_{n-1})$, so  we have $\depth(S/K)=1+(k-1)=k$  by induction,  Lemma \ref{sum}(\ref{sum-2}) and  Lemma \ref{complete}.
	Applying Lemma \ref{exact}(\ref{exact-2}) and Lemma \ref{decomposition}(\ref{decomposition-2}) to the  exact sequence (\ref{eqn:SES-1}),  we get $\depth(S/I_G)=k=1+|W|-T$, as wished.
	
	(ii)  If $T\ge 2$, i.e., $|W_1|\ge 2$, then $G \backslash N_G[w_1]$  has $\sum\limits_{i=2}^{k} |W_i|$ connected components, which in this case consist of  some cliques and isolated vertices in this case.
	Thus $\depth(S/(J+K))=\sum\limits_{i=2}^{k}|W_i|=|W|-T$ by  Lemma \ref{complete}. Meanwhile, $G\backslash w_1$ is the disjoint union $K_{a_{1,1}-1}\sqcup F_{k-1}^{W\backslash \{w_1\}}(K_{n-1})$, so we have $K=(x_{w_1})+I_{K_{a_{1,1}-1}}+I_{G\backslash V(K_{a_{1,1}})}$.
	By induction, Lemma \ref{sum}(\ref{sum-2}) and Lemma \ref{complete},  we obtain
	\begin{align*}
		\depth(S/K)=&1+[1+(T-1)+\sum\limits_{i=2}^{k}|W_i| -\max\{T-1, |W_2|,\ldots,|W_k|\}]\\
		=&1+\sum\limits_{i=1}^{k}|W_i|-\max\{T-1, |W_2|,\ldots,|W_k|\}\\
		\geq&1+|W|-T.
	\end{align*}
	Again applying Lemma \ref{exact}(\ref{exact-2}) and Lemma \ref{decomposition}(\ref{decomposition-2}) to the  exact sequence (\ref{eqn:SES-1}),  we can obtain  the desired result.
\end{proof}

\begin{Lemma}{\em (\cite[Lemma 3.5]{BBH})}
	\label{lem:induced-reg} Let $G$ be a simple graph and $H$ be its induced subgraph. Then $\reg(I_H)\le \reg(I_G)$.
\end{Lemma}

\begin{Theorem}
	\label{thm:fan-reg}
Let $G=F_k^W(K_n)$ be a $k$-fan graph of the complete graph $K_n$ on the set $W\subseteq [n]$ with $n\ge 2$  and $W=W_1\sqcup\cdots\sqcup W_k$ be a partition of $W$.  Suppose $p=|\{h_{i,j}: h_{i,j}\geq2\}|$. If $|W|\geq n-1$ and  $h_{i,r_i}\ge 2$ for each $i\in [k]$,  then $\reg(S/I_G)=p$. Otherwise, we have $\reg(S/I_G)=p+1$.
\end{Theorem}
\begin{proof}Since  $G$ is a chordal graph, we get  $\reg(S/I_{G})=\vartheta(G)$ by Lemma \ref{chordal}. 	Let $M=\{e_1,\ldots,e_{\vartheta(G)}\}$ be an induced matching  of $G$. It suffices to  show that $\vartheta(G)=p$ if $|W|\geq n-1$ and  $h_{i,r_i}\ge 2$ for each $i\in [k]$. Otherwise, $\vartheta(G)=p+1$. We consider the 	following three cases:
	
(1) If $|W|\geq n-1$ and  $h_{i,r_i}\ge 2$ for each $i\in [k]$, then $|V(K_n\backslash W)|\le 1$  and  $G\backslash W$ is  in this case the disjoint union of  $p$ cliques and some isolated vertices in this case.	 Let $K_{b_1},\dots, K_{b_p}$ be all the cliques in $G\backslash W$.	Since  the induced matching $M$ is  maximal, thus, for each $K_{b_i}$, there exists some $e\in M$ such that $e\in E(K_{b_i})$. Say $e_i\in E(K_{b_i})$ for $i=1,\ldots, p$.
	If $\vartheta(G)\ge p+1$, we set $e_{p+1}=\{u,v\}$, then $u,v\notin \cup_{i=1}^{p} V(K_{b_i})$ and $u,v$ cannot simultaneously belong to $V(K_n)$. Otherwise,  $u$ (or $v$)  and the two endpoints of some $e_i$  form a path with three vertices, which is contradictory.
	It follows that $u,v$ belong to the set of isolated vertices in $V(G\backslash W)$, which contradicts with $e_{p+1}=\{u,v\}\in E(G)$.
	
	(2)   If  $|W|\geq n-1$ and  $h_{i,r_i}=1$ for some $i\in [k]$, let $V(K_n)\cap V(K_{a_{i,r_i}})=\{w_{i,1},\ldots,w_{i,r_i}\}$, then $G\backslash (V(K_n)\backslash w_{i,r_i})$ is the disjoint union of $(p+1)$  cliques  and some isolated vertices.

(3) If $|W|\le n-2$, then  $|V(K_n\backslash W)|\ge 2$. In this case, $G\backslash W$ is the disjoint union of $(p+1)$  cliques  and some isolated vertices.

For cases (2) and (3), by arguments similar to (1), we obtain that $\vartheta(G)=p+1$ in these two cases, as claimed.
\end{proof}

\section{Results involving $*$ operations and $\circ$ operations}
In this section,  we will study some simple graphs obtained from the fan graphs  by using the  $*$ operation or the  $\circ$ operation. The main task of this section is
to provide some formulas for the depth and regularity of some graphs. We start by recalling from \cite{BMS} the two special gluing operations mentioned earlier.

\begin{Definition}
    \label{circ_*_operations}
    For $i= 1,2$, let $G_i$ be a graph with at least one leaf $f_i$ and $v_i$ be its neighbor with $\deg_{G_i}(v_i)\ge 2$.
     \begin{enumerate}%[a]
         \item Let $G$ be the graph obtained from $G_1$ and $G_2$ by
            identifying the vertices $f_1$ and $f_2$. In this case, we say that $G$
            is obtained from $G_1$ and $G_2$ by the \emph{$*$ operation} and
            write $G=(G_1,f_1)*(G_2,f_2)$ or simply $G= G_1 * G_2$. If we
            denote the identified vertex in $G$ by $f$, then we also write $G= G_1
            *_f G_2$.
            \item Let $H$ be the graph obtained from $G_1$ and $G_2$ by  first removing
            the leaves $f_1, f_2$, and then identifying the vertices
            $v_1$ and $v_2$.  In this case, we say that $H$ is obtained from $G_1$
            and $G_2$ by the \emph{$\circ$ operation} and write $H=(G_1,f_1)
            \circ (G_2,f_2)$ or simply $H=G_1 \circ G_2$. If $v_1$ and $v_2$
            are identified as the vertex $v$ in $G$, then we also write $H=
            G_1\circ_v G_2$.
    \end{enumerate}
\end{Definition}

\begin{Remark}
 From the above definition we can see that $G_1\sqcup G_2$ is a splitting graph of $G$ and $(G_1\backslash f_1)\sqcup (G_2\backslash f_2)$ is a splitting graph of $H$.
\end{Remark}

\begin{Example} \label{example2}
The following are examples of the graphs  obtained  from $G_1$ and $G_2$ by $\circ$ and  $*$ operations.
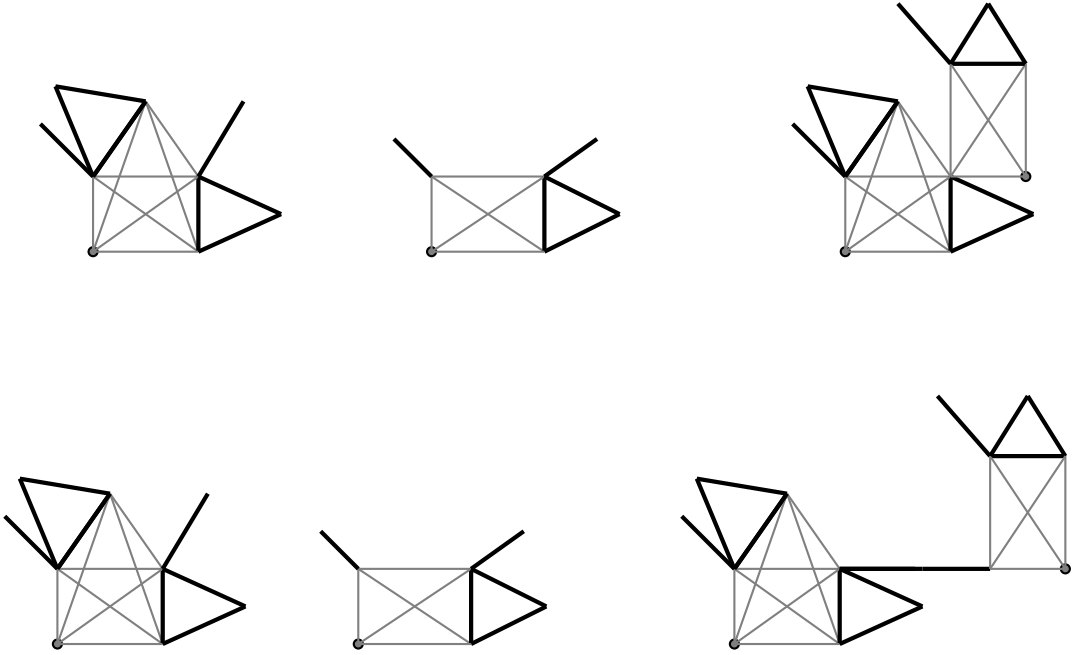
\begin{center}
		\begin{tikzpicture}[thick,>=stealth]
			%\draw[help lines] (-10,-10) grid (10,10);
			\setlength{\unitlength}{1mm}	
			\thicklines
%G_1 graph		    	
\put(-80,20){\circle*{1.5}}
\put(-73,30){\circle*{1.5}}
\put(-66,20){\circle*{1.5}}
\put(-66,10){\circle*{1.5}}
\draw[fill = gray] (-8,1) circle(0.06);
\put(-87,27){\circle*{1.5}}
\put(-85,32){\circle*{1.5}}
\put(-60,30){\circle*{1.5}}
\put(-55,15){\circle*{1.5}}

\draw[ultra thick](-8,2)--(-7.3,3);
 \draw[gray](-8,2)--(-6.6,2);
 \draw[gray](-8,2)--(-6.6,1);
 \draw[gray](-8,2)--(-8,1);
\draw[ultra thick](-8,2)--(-8.7,2.7);
\draw[ultra thick](-8,2)--(-8.5,3.2);

 \draw[gray](-7.3,3)--(-6.6,2);
 \draw[gray](-7.3,3)--(-6.6,1);
 \draw[gray](-7.3,3)--(-8,1);
\draw[ultra thick](-7.3,3)--(-8,2);
\draw[ultra thick](-7.3,3)--(-8.5,3.2);

\draw[ultra thick](-6.6,2)--(-6.6,1);
 \draw[gray](-6.6,2)--(-8,1);
\draw[ultra thick](-6.6,2)--(-6,3);
\draw[ultra thick](-6.6,2)--(-5.5,1.5);

 \draw[gray](-6.6,1)--(-8,1);	
\draw[ultra thick](-6.6,1)--(-5.5,1.5);	
\put(-74,3){$G_1$};
\put(-64,20){$v_1$};
\put(-62,32){$f_1$};

%G_2 graph
\put(-35,20){\circle*{1.5}}
\put(-20,20){\circle*{1.5}}
\put(-20,10){\circle*{1.5}}
\draw[fill = gray] (-3.5,1) circle(0.06);
\put(-40,25){\circle*{1.5}}
\put(-13,25){\circle*{1.5}}
\put(-10,15){\circle*{1.5}}

 \draw[gray](-3.5,2)--(-2,2);
 \draw[gray](-3.5,2)--(-2,1);
 \draw[gray](-3.5,2)--(-3.5,1);
\draw[ultra thick](-3.5,2)--(-4,2.5);

\draw[ultra thick](-2,2)--(-2,1);
 \draw[gray](-2,2)--(-3.5,1);
\draw[ultra thick](-2,2)--(-1.3,2.5);
\draw[ultra thick](-2,2)--(-1,1.5);

\draw[ultra thick](-2,1)--(-1,1.5);
 \draw[gray](-2,1)--(-3.5,1);
\put(-28,3){$G_2$};
\put(-36,22){$v_2$};
\put(-42,27){$f_2$};
\put(0,15){\vector(1,0){11}}

%G_1 \circ G_2		    	
\put(20,20){\circle*{1.5}}
\put(27,30){\circle*{1.5}}
\put(34,20){\circle*{1.5}}
\put(34,10){\circle*{1.5}}
\draw[fill = gray] (2,1) circle(0.06);
\put(13,27){\circle*{1.5}}
\put(15,32){\circle*{1.5}}
\put(45,15){\circle*{1.5}}
\put(44,35){\circle*{1.5}}
\draw[fill = gray] (4.4,2) circle(0.06);
\put(34,35){\circle*{1.5}}
\put(27,43){\circle*{1.5}}
\put(39,43){\circle*{1.5}}

\draw[ultra thick](2,2)--(2.7,3);
 \draw[gray](2,2)--(3.4,2);
 \draw[gray](2,2)--(3.4,1);
 \draw[gray](2,2)--(2,1);
\draw[ultra thick](2,2)--(1.3,2.7);
\draw[ultra thick](2,2)--(1.5,3.2);

 \draw[gray](2.7,3)--(3.4,2);
 \draw[gray](2.7,3)--(3.4,1);
 \draw[gray](2.7,3)--(2,1);
\draw[ultra thick](2.7,3)--(2,2);
\draw[ultra thick](2.7,3)--(1.5,3.2);

\draw[ultra thick](3.4,2)--(3.4,1);
 \draw[gray](3.4,2)--(2,1);
\draw[ultra thick](3.4,2)--(4.5,1.5);
 \draw[gray](3.4,2)--(3.4,3.5);
 \draw[gray](3.4,2)--(4.4,3.5);
 \draw[gray](3.4,2)--(4.4,2);

 \draw[gray](3.4,1)--(2,1);	
\draw[ultra thick](3.4,1)--(4.5,1.5);	
\draw[ultra thick](3.4,3.5)--(4.4,3.5);	
 \draw[gray](3.4,3.5)--(4.4,2);
\draw[ultra thick](3.4,3.5)--(2.7,4.3);
\draw[ultra thick](3.4,3.5)--(3.9,4.3);

 \draw[gray](4.4,3.5)--(4.4,2);
\draw[ultra thick](4.4,3.5)--(3.9,4.3);
\put(27,3){$G_1 \circ G_2$};
\put(36,21){$v$};
\put(-35,-5){$Figure \;2$\hspace{0.3cm} $\circ$\hspace{0.1cm} $operation$};
\end{tikzpicture}
	\end{center}

\vspace{1.5cm}
	\begin{center}
		\begin{tikzpicture}[thick,>=stealth]
			%\draw[help lines] (-10,-10) grid (10,10);
			\setlength{\unitlength}{1mm}
				\thicklines
	
%G_1 graph		    	
\put(-80,20){\circle*{1.5}}
\put(-73,30){\circle*{1.5}}
\put(-66,20){\circle*{1.5}}
\put(-66,10){\circle*{1.5}}
\draw[fill = gray] (-8,1) circle(0.06);
\put(-87,27){\circle*{1.5}}
\put(-85,32){\circle*{1.5}}
\put(-60,30){\circle*{1.5}}
\put(-55,15){\circle*{1.5}}

\draw[ultra thick](-8,2)--(-7.3,3);
 \draw[gray](-8,2)--(-6.6,2);
 \draw[gray](-8,2)--(-6.6,1);
 \draw[gray](-8,2)--(-8,1);
\draw[ultra thick](-8,2)--(-8.7,2.7);
\draw[ultra thick](-8,2)--(-8.5,3.2);

 \draw[gray](-7.3,3)--(-6.6,2);
 \draw[gray](-7.3,3)--(-6.6,1);
 \draw[gray](-7.3,3)--(-8,1);
\draw[ultra thick](-7.3,3)--(-8,2);
\draw[ultra thick](-7.3,3)--(-8.5,3.2);

\draw[ultra thick](-6.6,2)--(-6.6,1);
 \draw[gray](-6.6,2)--(-8,1);
\draw[ultra thick](-6.6,2)--(-6,3);
\draw[ultra thick](-6.6,2)--(-5.5,1.5);

 \draw[gray](-6.6,1)--(-8,1);	
\draw[ultra thick](-6.6,1)--(-5.5,1.5);	
\put(-74,3){$G_1$};
\put(-64,20){$v_1$};
\put(-62,32){$f_1$};

%G_2 graph
\put(-40,20){\circle*{1.5}}
\put(-25,20){\circle*{1.5}}
\put(-25,10){\circle*{1.5}}
\draw[fill = gray] (-4,1) circle(0.06);
\put(-45,25){\circle*{1.5}}
\put(-18,25){\circle*{1.5}}
\put(-15,15){\circle*{1.5}}

 \draw[gray](-4,2)--(-2.5,2);
 \draw[gray](-4,2)--(-2.5,1);
 \draw[gray](-4,2)--(-4,1);
\draw[ultra thick](-4,2)--(-4.5,2.5);

\draw[ultra thick](-2.5,2)--(-2.5,1);
 \draw[gray](-2.5,2)--(-4,1);
\draw[ultra thick](-2.5,2)--(-1.8,2.5);
\draw[ultra thick](-2.5,2)--(-1.5,1.5);

\draw[ultra thick](-2.5,1)--(-1.5,1.5);
 \draw[gray](-2.5,1)--(-4,1);
\put(-28,3){$G_2$};
\put(-41,22){$v_2$};
\put(-47,27){$f_2$};
\put(-10,15){\vector(1,0){11}}

%G_1 * G_2		    	
\put(10,20){\circle*{1.5}}
\put(17,30){\circle*{1.5}}
\put(24,20){\circle*{1.5}}
\put(24,10){\circle*{1.5}}
\draw[fill = gray] (1,1) circle(0.06);
\put(3,27){\circle*{1.5}}
\put(5,32){\circle*{1.5}}
\put(35,15){\circle*{1.5}}
\put(35,20){\circle*{1.5}}
\put(44,20){\circle*{1.5}}

\put(44,35){\circle*{1.5}}
\put(44,20){\circle*{1.5}}
\draw[fill = gray] (5.4,2) circle(0.06);
\put(54,35){\circle*{1.5}}
\put(49,43){\circle*{1.5}}
\put(37,43){\circle*{1.5}}

\draw[ultra thick](1,2)--(1.7,3);
 \draw[gray](1,2)--(2.4,2);
 \draw[gray](1,2)--(2.4,1);
 \draw[gray](1,2)--(1,1);
\draw[ultra thick](1,2)--(0.3,2.7);
\draw[ultra thick](1,2)--(0.5,3.2);

 \draw[gray](1.7,3)--(2.4,2);
 \draw[gray](1.7,3)--(2.4,1);
 \draw[gray](1.7,3)--(1,1);
\draw[ultra thick](1.7,3)--(1,2);
\draw[ultra thick](1.7,3)--(0.5,3.2);

\draw[ultra thick](2.4,2)--(2.4,1);
 \draw[gray](2.4,2)--(1,1);
\draw[ultra thick](2.4,2)--(3.5,1.5);

\draw[ultra thick](2.4,2)--(3.4,2);
\draw[ultra thick](2.4,2)--(3.5,2);

\draw[ultra thick](3.5,2)--(4.4,2);

 \draw[gray](2.4,1)--(1,1);	
\draw[ultra thick](2.4,1)--(3.5,1.5);	

\draw[ultra thick](4.4,3.5)--(5.4,3.5);	
 \draw[gray](4.4,3.5)--(5.4,2);
\draw[ultra thick](4.4,3.5)--(3.7,4.3);
\draw[ultra thick](4.4,3.5)--(4.9,4.3);
 \draw[gray](4.4,3.5)--(4.4,2);

 \draw[gray](5.4,3.5)--(5.4,2);
\draw[ultra thick](5.4,3.5)--(4.9,4.3);
 \draw[gray](5.4,3.5)--(4.4,2);

 \draw[gray](5.4,2)--(4.4,2);
\put(27,3){$G_1 * G_2$};

\put(24,22){$v_1$};
\put(34,22){$f$};
\put(42,17){$v_2$};
\put(-35,-5){$Figure \;3$\hspace{0.3cm} $*$\hspace{0.1cm} $operation$};

		\end{tikzpicture}
	\end{center}
\end{Example}

\vspace{0.8cm}
\subsection{Glue via the $\circ$ operation}
First of all, we compute   the depth and regularity of  the quotient ring  of the edge ideal of a graph, which is obtained  by a single $\circ$ operation.

\begin{Lemma}\label{lower bound}
		Suppose that $G=(G_1,f_1)\circ (G_2,f_2)$,  where $G_i=F_{k_i}^{W_i}(K_{n_i})$ is a $k_i$-fan graph of $K_{n_i}$ on the set $W_i$ with $n_i\ge 2$ and $k_i\ge 1$ for $i\in[2]$. Then we have
\[
\depth(S_G/I_{G})\geq \depth(S_{G_1}/I_{G_1})+\depth(S_{G_2}/I_{G_2})-2.
\]
	\end{Lemma}

\begin{proof}
	Let $v_i$ be the neighbor point of $f_i$ in $G_i$ and $W_i=W_{i,1}\sqcup\cdots\sqcup W_{i,k_i}$ be a partition of $W_i$ with $v_i\in W_{i,1}$ for $i\in[2]$. Suppose that $v_1$ and $v_2$ are identified with $v$ in $G$ by the $\circ$ operation. Let $J=(x_{N_G(v)})+I_{G \setminus N_G[v]}$ and $K=(x_{v})+I_{G\backslash v}$, then
$G\backslash N_G[v]$  has $\sum\limits_{i=1}^{2}(|W_i|-|W_{i,1}|)$ connected components consisting of cliques and isolated vertices. Let $T_i=\max\{|W_{i,1}|, |W_{i,2}|,\ldots,|W_{i,k_i}|\}$ for $i\in[2]$, then  $|W_{i,1}|\le T_i$.
Thus $\depth(S_G/J)=1+\sum\limits_{i=1}^{2}(|W_i|-|W_{i,1}|)\ge 1+\sum\limits_{i=1}^{2}(|W_i|-T_i)$ by and Lemma \ref{complete}.
Meanwhile, $G\backslash v=(G_1\backslash \{v,f_1\})\sqcup (G_2\backslash \{v,f_2\})$ and every $G_i\backslash \{v,f_i\}$  is a fan graph, so  we have $K=(x_{v})+I_{G_1\backslash \{v,f_1\}}+I_{G_2\backslash \{v,f_2\}}$. Then, by Theorem \ref{depth of F_k}, we get
\begin{align*}
	\depth(S_G/K)=&\depth(S_{G_1\backslash \{v,f_1\}}/I_{G_1\backslash \{v,f_1\}})+\depth(S_{G_2\backslash \{v,f_2\}}/I_{G_2\backslash \{v,f_2\}})\\
	=&[1+(|W_{1}|-1)-T'_1]+[1+(|W_{2}|-1)-T'_2]\\
	=&(1+|W_1|-T'_1)+(1+|W_2|-T'_2)-2\\
	\geq&\depth(S_{G_1}/I_{G_1})+\depth(S_{G_2}/I_{G_2})-2.
	\end{align*}	
where $T'_i=\max\{|W_{i,1}|-1, |W_{i,2}|,\ldots,|W_{i,k_i}|\}$ for $i\in[2]$, and the last inequality holds because of every $T'_i\le T_i$ and Theorem
\ref{depth of F_k}.
	Again applying Lemma \ref{exact}(\ref{exact-2}) and Lemma \ref{decomposition} to the  exact sequence (\ref{eqn:SES-1}),  we obtain the desired result.
\end{proof}

Jayanthan et al. in \cite{JKS2} introduced the notion of clique sum. Here, we make a
simple application of it. By a clique sum $G_1\cup_v  G_2$, we mean a union of graphs $G_1$ and $G_2$ such
that $V(G_1)\cap V(G_2)=\{v\}$.
\begin{Lemma}
\label{subgraph}
		Let  $G=F_k^W(K_n)$ be a $k$-fan graph of  $K_n$ on the set $W$ with $n\ge 2$ and $k\ge 1$. Suppose $f$ is a leaf of $G$,  $v$ is its neighbor point, and $W=W_1\sqcup\cdots\sqcup W_k$ is a partition of $W$ with  $v\in W_1$. Let $T=\max\{|W_{1}|, |W_{2}|,\ldots,|W_{k}|\}$, $T'=\max\{|W_{1}|-1, |W_{2}|,\ldots,|W_{k}|\}$ with $T'= T-1$.
 Then $\depth(S_{G\backslash f}/I_{G\backslash f})=\depth(S_G/I_G)$.
	\end{Lemma}
	\begin{proof} Let $G'=G\backslash f$. Choose $J=(x_{N_{G'}(v)})+I_{{G'}\backslash N_{G'}[v]}$ and $K=(x_{v})+I_{{G'}\backslash v}$,  then $J+K=(x_{N_{G'}[v]})+I_{G \backslash N_{G'}[v]}$   by Lemma \ref{decomposition}(\ref{decomposition-1}) and  $G'\backslash N_{G'}[v]$ has $(|W|-|W_1|)$ connected components  consisting of cliques and isolated vertices. Thus $\depth(S_{G'}/(J+K))=|W|-|W_1|=\depth(S_{G}/I_{G})-1$ by  Lemma \ref{complete}. Meanwhile, $G'\backslash v$ is a fan graph of $K_{n-1}$ on the set $W\backslash v$, so $\depth(S_{G'}/K)=1+(|W|-1)-T'=\depth(S_G/I_{G})$. Applying Lemma \ref{exact}(\ref{exact-2}) and Lemma \ref{decomposition} to the  exact sequence (\ref{eqn:SES-1}) by substituting $S_{G'}$ for $S$,  we get the desired result.
	\end{proof}

Now  we are ready to prove the first major result of this subsection.
\begin{Theorem}
	\label{thm:depth_Fan_Fan_circ}
	Suppose that $G=(G_1,f_1)\circ (G_2,f_2)$,  where $G_i=F_{k_i}^{W_i}(K_{n_i})$ is a $k_i$-fan graph of $K_{n_i}$ on the set $W_i$ with $n_i\ge 2$ and $k_i\ge 1$ for $i\in[2]$.
 Let $v_i$ be the neighbor of $f_i$ in $G_i$ and $W_i=W_{i,1}\sqcup\cdots\sqcup W_{i,k_i}$ be a partition of $W_i$ with $v_i\in W_{i,1}$ for $i\in[2]$. Let  $T_i=\max\{|W_{i,1}|, |W_{i,2}|,\ldots,|W_{i,k_i}|\}$ and  $T'_i=\max\{|W_{i,1}|-1, |W_{i,2}|,\ldots,|W_{i,k_i}|\}$ for $i\in[2]$, and $t=|\{i:T'_i\ne T_i-1\}|$. Then
\[
\depth(S_G/I_{G})=\depth(S_{G_1}/I_{G_1})+\depth(S_{G_2}/I_{G_2})-s
\]
where $s=1$ if $t\le 1$, otherwise, $s=2$.
\end{Theorem}
\begin{proof}	Let $v=v_1=v_2$. By the definition of $t$, we have $t\in\{0,1,2\}$. We distinguish the following three cases:
	
(I) If $t=0$, then  $T'_i=T_i-1$ for $i\in[2]$. This implies that each $T_i=|W_{i,1}|$ and $|W_{i,1}|>|W_{i,j}|$ for  $j=2,\ldots,k_i$. In this case,  we set $J=(x_{N_G(v)})+I_{G\backslash N_G[v]}$ and $K=(x_{v})+I_{G\backslash v}$,  then $J+K=(x_{N_G[v]})+I_{G \backslash N_G[v]}$ and $G\backslash N_G[v]$  has  $\sum\limits_{i=1}^{2}(|W_i|-|W_{i,1}|)$ connected components consisting of some cliques and isolated vertices.
Thus $\depth(S_G/(J+K))=\sum\limits_{i=1}^{2}(|W_i|-|W_{i,1}|)=\depth(S_{G_1}/I_{G_1})+\depth(S_{G_2}/I_{G_2})-2$ by  Lemma \ref{complete} and Theorem
\ref{depth of F_k}.
Meanwhile, $G\backslash v=(G_1\backslash \{v,f_1\})\sqcup (G_2\backslash \{v,f_2\})$ and every $G_i\backslash \{v,f_i\}$ is a fan graph. So  we have $K=(x_{v})+I_{G_1\backslash \{v,f_1\}}+I_{G_2\backslash \{v,f_2\}}$. By Theorem \ref{depth of F_k}, we get
	\begin{align*}
	\depth(S_G/K)=&\depth(S_{G_1 \backslash \{v,f_1\}}/I_{G_1\backslash \{v,f_1\}})+\depth(S_{G_2\backslash \{v,f_2\}}/I_{G_2\backslash \{v,f_2\}})\\
	=&[1+(|W_{1}|-1)-T'_1]+[1+(|W_{2}|-1)-T'_2]\\
	\end{align*}
\begin{align*}
=&[1+(|W_{1}|-1)-(T_1-1)]+[1+(|W_{2}|-1)-(T_2-1)]\\
=&(1+|W_1|-T_1)+(1+|W_2|-T_2)\\
	=&\depth(S_{G_1}/I_{G_1})+\depth(S_{G_2}/I_{G_2}).
\end{align*}		
By Lemma \ref{exact}(\ref{exact-2}) and Lemma \ref{decomposition}(\ref{decomposition-2}) and the  exact sequence (\ref{eqn:SES-1}),  we get the desired result.

(II) If $t=1$,  we assume that  $T'_1=T_1-1$ and $T'_2\ne T_2-1$. It follows that $T_1=|W_{1,1}|>|W_{1,j}|$ for  $j=2,\ldots,k_1$ and $T_2=|W_{2,j}|\ge |W_{2,1}|$ for some $j\ne 1$. Say $T_2=|W_{2,2}|$. In this case, we prove   by induction on $n_2$ that
\[
\depth(S_G/I_{G})=\depth(S_{G_1}/I_{G_1})+\depth(S_{G_2}/I_{G_2})-1.
\]
If $n_2=2$, then $G_2$ can only be  $G_2=P_3\cup_u  K_b$, which is the  clique sum of $P_3$ and some $K_b$. So we have  $\depth(S_{G_2}/I_{G_2})=2$ by Theorem \ref{depth of F_k}. In this case, let $J=(x_{N_G(u)})+I_{G\backslash N_G[u]}$ and $K=(x_{u})+I_{G\backslash u}$, then $J+K=(x_{N_G[u]})+I_{G \backslash N_G[u]}$,  $G\backslash N_G[u]=G_1\backslash \{v_1,f_1\}$ is a fan graph and
$G\backslash u=(G_1\backslash f_1)\sqcup K_{b-1}$. Thus
 \begin{align*}
\depth(S_G/(J+K))&=\depth(S_{G_1\backslash \{v_1,f_1\}}/I_{G_1\backslash \{v_1,f_1\}})\\
&=1+(|W_1|-1)-T'_1=\depth(S_{G_1}/I_{G_1})
\end{align*}
and
\[
\depth(S_G/K)=1+\depth(S_{G_1\backslash f_1}/I_{G_1\backslash f_1})=1+\depth(S_{G_1}/I_{G_1})
\]
 by Lemma \ref{subgraph}. The  desired results follow  from Lemma \ref{exact}(\ref{exact-2}), Lemma \ref{decomposition}(\ref{decomposition-2}) and the  exact sequence (\ref{eqn:SES-1})  by substituting $S_{G}$ for $S$.

 Now suppose $n_2\ge 3$. In this case, let $W_{2,2}=\{w_{1},\ldots,w_{r_2}\}$ and $\{K_{a_{2,1}},\ldots,K_{a_{2,r_2}}\}$ be a branch of the fan on $W_{2,2}$ such that $V(K_{n_2})\cap V(K_{a_2,j})=\{w_1,\ldots,w_j\}$ for $j\in [r_2]$. Choose  $J=(x_{N_G(w_1)})+I_{G\backslash N_G[w_1]}$ and $K=(x_{w_1})+I_{G\backslash w_1}$, then $J+K=(x_{N_G[w_1]})+I_{G \backslash N_G[w_1]}$,  $G\backslash N_G[w_1]$ is the disjoint union of the graph $G_1\backslash \{v_1,f_1\}$ and  $(|W_2|-|W_{2,2}|-1)$ connected components  consisting of some cliques and isolated vertices.
 Then we obtain
 \begin{align}
 	\depth(S_G/(J+K))&=\depth(S_{G_1\backslash \{v_1,f_1\}}/I_{G_1\setminus \{v_1,f_1\}})+(|W_2|-|W_{2,2}|-1)\notag\\
 	&=[1+(|W_1|-1)-T'_1]+(1+|W_2|-|W_{2,2}|-2)\notag\\
 	&=\depth(S_{G_1}/I_{G_1})+\depth(S_{G_2}/I_{G_2})-2.
 	\label{eqn:SES-2}
 \end{align}
Meanwhile, $G\backslash w_1=G_1\circ (G_2\setminus  w_1)$ and  $G_2\backslash  w_1=F_{q}^{W_2\backslash \{w_1\}}(K_{n_2-1})$ for some $q$. In this case, to compute $\depth(S_{G\backslash w_1}/I_{G\backslash w_1})$, we set $T^*_1=T_1$, $T'^*_1=T'_1$, $T^{*}_2=\max\{|W_{2,1}|, |W_{2,2}|-1,|W_{2,3}|,\ldots,|W_{2,k_2}|\}$ and  $T'^*_2=\max\{|W_{2,1}|-1, |W_{2,2}|-1, |W_{2,3}|,\\ \ldots,|W_{2,k_2}|\}$. Let $t'=|\{i:T'^*_i\ne T^*_i-1\}|$, then $t'\leq 1$. From  the above case (I) and by induction  we have
 \[
\depth(S_G/K)=1+ \depth(S_{G\backslash w_1}/I_{G\backslash w_1})=\depth(S_{G_1}/I_{G_1})+\depth(S_{G_2\backslash w_1}/I_{G_2\backslash w_1}).
 \]
%It follows that
%\begin{align}
%	\depth(S_G/K)=&1+\depth(S_{G\backslash w_1}/I_{G\backslash w_1})\notag\\
%=&1+(\depth(S_{G_1}/I_{G_1})+\depth(S_{G_2\backslash w_1}/I_{G_2\backslash w_1})-1)\label{eqn:SES-3}\\
%=&\depth(S_{G_1}/I_{G_1})+\depth(S_{G_2\setminus w_1}/I_{G_2\backslash w_1})\notag\\
%=&\depth(S_{G_1}/I_{G_1})+[1+(|W_2|-1)-T^*_2]\notag\\
%\ge &\depth(S_{G_1}/I_{G_1})+[1+(|W_2|-1)-T_2]\notag\\
%=&\depth(S_{G_1}/I_{G_1})+\depth(S_{G_2}/I_{G_2})-1\notag
%\end{align}	
%where the inequality holds since  $T^*_2\le T_2$.
The desired result follows from  Lemma \ref{exact}(\ref{exact-2}), the equality (\ref{eqn:SES-2}),  Lemma \ref{decomposition} and the  exact sequence (\ref{eqn:SES-1}).

(III) If $t=2$, then $T'_i\ne T_i-1$ for $i\in [2]$. This implies that $T_i=|W_{i,j_i}|\ge |W_{i,1}|$ for some $j_i\ne 1$. Suppose $T_i=|W_{i,2}|$  for $i\in [2]$. In this case, let $W_{2,2}=\{w_{1},\ldots,w_{r_2}\}$ and $\{K_{a_{2,1}},\ldots,K_{a_{2,r_2}}\}$ be a branch of the fan on $W_{2,2}$ with $V(K_{n_2})\cap V(K_{a_2,j})=\{w_1,\ldots,w_j\}$ for $j\in [r_2]$. Choose  $J=(x_{N_G(w_1)})+I_{G\backslash N_G[w_1]}$ and $K=(x_{w_1})+I_{G\backslash w_1}$, then $J+K=(x_{N_G[w_1]})+I_{G \backslash N_G[w_1]}$. Similar to the case $t=1$,  $G\backslash N_G[w_1]$ is the disjoint union of a graph $G_1\backslash \{v_1,f_1\}$ and  $(|W_2|-|W_{2,2}|-1)$ connected components  consisting of some cliques and isolated vertices, and $G\backslash  w_1=G_1\circ (G_2\backslash   w_1)$ and $G_2\backslash   w_1=F_{q}^{W_2\backslash \{w_1\}}(K_{n_2-1})$ for some $q$.  So we have
\begin{align}
	\depth(S_G/(J+K))&=\depth(S_{G_1\backslash \{v_1,f_1\}}/I_{G_1\backslash \{v_1,f_1\}})+(|W_2|-|W_{2,2}|-1)\notag\\
	&=[1+(|W_1|-1)-T'_1]+(1+|W_2|-|W_{2,2}|-2)\label{eqn:SES-4}\\
	&=[1+(|W_1|-1)-T_1]+(1+|W_2|-|W_{2,2}|-2)\notag\\
	&=\depth(S_{G_1}/I_{G_1})+\depth(S_{G_2}/I_{G_2})-3.\notag	
\end{align}
and
\begin{align}
	\depth(S/K)=&1+\depth(S_{G\backslash  w_1}/I_{G\backslash  w_1})\notag\\
	\ge &1+\depth(S_{G_1}/I_{G_1})+\depth(S_{G_2\backslash  w_1}/I_{G_2\backslash  w_1})-2\notag\\
 =&\depth(S_{G_1}/I_{G_1})+[1+(|W_2|-1)-T^*_2]-1\notag\\
	\ge &\depth(S_{G_1}/I_{G_1})+(1+|W_2|-T_2)-2\notag\\
	=&\depth(S_{G_1}/I_{G_1})+\depth(S_{G_2}/I_{G_2})-2.
	\label{eqn:SES-5}
\end{align}	
where the first inequality holds by Lemma \ref{lower bound} and the second inequality holds by $T^*_2\le T_2$, where $T^{*}_2=\max\{|W_{2,1}|, |W_{2,2}|-1,|W_{2,3}|,\ldots,|W_{2,k_2}|\}$. By Lemma \ref{exact}(\ref{exact-2}), the relations (\ref{eqn:SES-4}), (\ref{eqn:SES-5}), Lemma \ref{decomposition}(\ref{decomposition-2}) and the  exact sequence (\ref{eqn:SES-1}), we get the wished result.
\end{proof}

\begin{Lemma}\label{clique sum}
	Let  $G_{1}=F_{k}^{W}(K_{n})$ be a $k$-fan graph of $K_{n}$ on the set $W\subsetneq [n]$ with $n\ge 2$ and $P_2$ be a path  with $2$ vertices.
Suppose that $G=G_1\cup_v P_2$ is the clique sum of  graphs $G_1$ and $P_2$ with  $V(G_1)\cap V(P_2)=\{v\}$ and  $v\in V(K_{n})\backslash W$. Let $W=W_1\sqcup\cdots\sqcup W_k$ be a partition of $W$. Then
\[
\reg(S_{G\backslash v}/I_{G\backslash v})=\reg(S_G/I_G)-s
\]
where $s=1$ if $|W|\ge  n-2$ and  $h_{i,r_i} \ge 2$ for all $i\in [k]$, otherwise, $s=0$.
\end{Lemma}

\begin{proof} $G$ is actually a $(k+1)$-fan graph of $K_{n}$ on the set $W\cup \{v\}$ and $G\backslash v$ is the disjoint union of a graph $F_{k}^{W}(K_{n-1})$ and an isolated vertex. Suppose $p=|\{h_{i,j}: h_{i,j}\geq2\}|$. By  Lemma  \ref{chordal} and Theorem \ref{thm:fan-reg}, we getthat if $|W|\ge  n-2$ and  $h_{i,r_i} \ge 2$ for all $i\in [k]$, then  $\reg(S_G/I_G)=p+1$ and $\reg(S_{G\backslash v}/I_{G\backslash v})=p$. Otherwise,  $\reg(S_G/I_G)=\vartheta(G\backslash v)=p+1$,  as expected.
\end{proof}

Next, we are ready to prove another major result of this subsection.
\begin{Theorem}
	\label{thm:reg_Fan_Fan_circ}
 	Suppose that $G=(G_1,f_1)\circ (G_2,f_2)$,  where $G_i=F_{k_i}^{W_i}(K_{n_i})$ is a $k_i$-fan graph of $K_{n_i}$ on the set $W_i$ with $n_i\ge 2$ and $k_i\ge 1$ for $i\in[2]$.
 Let $v_i$ be the neighbor point of $f_i$ in $G_i$ and $t=|\{i\in [2]: \reg(S_{G_i\backslash v_i}/I_{G_i\backslash v_i})\ne \reg(S_{G_i}/I_{G_i})\}|$. We have
 \begin{itemize}
 	\item[(1)] If $t\leq 1$, then $\reg(S_G/I_G)=\reg(S_{G_1}/I_{G_1})+\reg(S_{G_2}/I_{G_2})-t$,
 	\item[(2)] If $t=2$ and $|W_i|=n_i-1$ for some $i\in [2]$, then
 	\[
 	\reg(S_G/I_G)=\reg(S_{G_1}/I_{G_1})+\reg(S_{G_2}/I_{G_2})-1,
 	\]
 	\item[(3)]  If $t=2$ and $|W_i|=n_i$ for all $i\in [2]$, then
 	\[
 	\reg(S_G/I_G)=\reg(S_{G_1}/I_{G_1})+\reg(S_{G_2}/I_{G_2})-2.
 	\]
 \end{itemize}	
\end{Theorem}
\begin{proof} From the definition of $t$, we get  $t\in\{0,1,2\}$. Meanwhile, we  have $\reg(S_G/I_{G})=\vartheta(G)$ and $\reg(S_{G_i}/I_{G_i})=\vartheta(G_i)$ for $i\in [2]$ by Lemma \ref{chordal},  where $\vartheta(G)$ is the induced matching number of $G$.	
  We divide into the following three cases:

(1) If $t=0$,  then $\reg(S_{G_i\backslash v_i}/I_{G_i\backslash v_i})=\reg(S_{G_i}/I_{G_i})$ for $i\in[2]$, i.e., every $\vartheta(G_i\backslash v_i)=\vartheta(G_i)$. In this case, it is clear that $M=M_1 \sqcup M_2$ is an induced matching of $G$, where every $M_i=\{e_{i,1},\ldots,e_{i,\vartheta(G_i\backslash v_i)}\}$ is an induced matching  of $G_i\backslash v_i$. So $\vartheta(G)\ge \vartheta(G_1\backslash v_1)+\vartheta(G_2\backslash v_2)=\vartheta(G_1)+\vartheta(G_2)$.  If $\vartheta(G)\geq \vartheta(G_1)+\vartheta(G_2)+1$. Suppose $M'=\{e_1,\ldots,e_{\vartheta(G_1)+\vartheta(G_2)+1}\}$ is an induced matching of $G$, then there exist at least $(\vartheta(G_1)+1)$ elements in $M'$ belonging to $E(G_1\backslash f_1)$ or  $(\vartheta(G_2)+1)$ elements in $M'$ belonging to $E(G_2\backslash f_2)$. This forces $\vartheta(G_1)+1\le \vartheta(G_1\backslash f_1)\le \vartheta(G_1)$ or $\vartheta(G_2)+1\le \vartheta(G_2\backslash f_2)\le \vartheta(G_2)$, a contradiction.

(2) If $t=1$. By symmetry, we assume  that $\reg(S_{G_1\backslash v_1}/I_{G_1\backslash v_1})=\reg(S_{G_1}/I_{G_1})$ and $\reg(S_{G_2\backslash v_2}/I_{G_2\backslash v_2})\ne\reg(S_{G_2}/I_{G_2})$, i.e., $\vartheta(G_1\backslash v_1)=\vartheta(G_1)$ and $\vartheta(G_2\backslash v_2)\ne\vartheta(G_2)$. In this case, let $M_i=\{e_{i,1},\ldots,e_{i,\vartheta(G_i\backslash v_i)}\}$ be an induced matching  of $G_i\backslash v_i$ for $i\in [2]$,
then $M=M_1 \sqcup M_2$ is an induced matching of $G$. Thus $\vartheta(G)\ge \vartheta(G_1\backslash v_1)+\vartheta(G_2\backslash v_2)=\vartheta(G_1)+\vartheta(G_2)-1$, since $\vartheta(G_2\backslash v_2)=\vartheta(G_2)-1$. If $\vartheta(G)\geq \vartheta(G_1)+\vartheta(G_2)$ and $M''=\{e_1,\ldots,e_{\vartheta(G_1)+\vartheta(G_2)}\}$ is an induced matching of $G$, then,  by the assumption that $\vartheta(G_1\backslash v_1)=\vartheta(G_1)$ and $\vartheta(G_2\backslash v_2)=\vartheta(G_2)-1$, we obtain that  there exists at least $(\vartheta(G_1)+1)$ elements in $M''$ belong to $E(G_1\backslash f_1)$ or  $\vartheta(G_2)$ elements in $M''$ belong to $E(G_2\backslash f_2)$,  implying $\vartheta(G_1)+1\le \vartheta(G_1\backslash f_1)$ or $\vartheta(G_2)\le \vartheta(G_2\backslash f_2)$. We consider  two subcases:

(i) If $\vartheta(G_1)+1\le \vartheta(G_1\backslash f_1)$, then $\vartheta(G_1)+1\le \vartheta(G_1)$ since $\vartheta(G_1\backslash f_1)\le \vartheta(G_1)$, a contradiction.

(ii) If $\vartheta(G_2)\le \vartheta(G_2\backslash f_2)$, then $\vartheta(G_2\backslash f_2)=\vartheta(G_2)$, since  $\vartheta(G_2\backslash f_2)\le \vartheta(G_2)$. In this case, there exist $\vartheta(G_1)$ elements in  $M''$, say $e_{i_1},\ldots, e_{i_{\vartheta(G_1)}}$,  consisting  of   an induced matching  of $G_1\backslash f_1$, and the remaining elements  consist   of   an induced matching  of $G_2\backslash f_2$. Note that  $\vartheta(G_2)=\vartheta(G_2\backslash v_2)+1$, there exists some $e$  with $v_2$ as its an  endpoint belonging to $M''\backslash \{e_{i_1},\ldots, e_{i_{\vartheta(G_1)}}\}$. Let $e'=\{v,f_1\}\in E(G_1)$ where $v=v_1=v_2$, then  $\{e_{i_1},\ldots, e_{i_{\vartheta(G_1)}}, e'\}$ is
  an induced matching  of $G_1$, a  contradiction.

(3) If $t=2$, then $\reg(S_{G_i\backslash v_i}/I_{G_i\backslash v_i})\ne\reg(S_{G_i}/I_{G_i})$ for $i\in[2]$, i.e., every $\vartheta(G_i\backslash v_i)\ne\vartheta(G_i)$. Let  $W_i=W_{i_1}\sqcup\cdots\sqcup W_{i_{k_i}}$ be a partition of $W_i$  with $v_i\in W_{i_1}$. From the assumptions that for each $i$, $\reg(S_{G_i\backslash v_i}/I_{G_i\backslash v_i})\ne\reg(S_{G_i}/I_{G_i})$ and  $f_i$ is a leaf of $G_i$ with a unique neighbor point $v_i$,
we have  $|W_{i_1}|=1$ from the proof of Theorem \ref{thm:fan-reg}.
Let $W_{i_\ell}=\{w_{i_\ell,1},\ldots,w_{i_\ell,r_{i_\ell}}\}$ for $\ell\in[k_i]$ and $i\in [2]$, then
by Lemma \ref{clique sum} we get that  $|W_i|\ge  n_i-1$ and   $h_{i_\ell,r_{i_\ell}} \ge 2$ for all $\ell=2,\ldots,k_i$ and $i\in [2]$.
In this case, each
$G_i\backslash W_i$ is the disjoint union of $(\vartheta(G_i)-1)$ cliques $K'_{i,1}$, \ldots,$K'_{i,\vartheta(G_i)-1}$ and some isolated vertices. Let $M_i=\{e_{ij} | e_{ij}\in E(K'_{i,j}), j\in[\vartheta(G_i)-1]\}$ for $i\in [2]$, then  $M_i$ is an induced matching  of $G_i\backslash v_i$.  We consider two subcases:

(i) If  $|W_i|=n_i-1$ for some $i\in [2]$. Suppose that  $|W_1|=n_1-1$. In this case, let $M=M_1\sqcup M_2\sqcup\{e\}$, where edge  $e=\{v_1,u\}$ with $u\in V(K_{n_1})\backslash W_1$. Then $M$   is  an induced matching of $G$ by the choice of $M_1$ and $M_2$. So $\vartheta(G)\ge \vartheta(G_1)+\vartheta(G_2)-1$. If $\vartheta(G)\geq \vartheta(G_1)+\vartheta(G_2)$ and  $M'''=\{e_1,\ldots,e_{\vartheta(G_1)+\vartheta(G_2)}\}$ is an induced matching of $G$, then  there exist at least $\vartheta(G_1)$ elements in $M'''$ belonging to $E(G_1\backslash f_1)$ and $\vartheta(G_2)$ elements in $M'''$ belonging to $E(G_2\backslash f_2)$,  implying $\vartheta(G_1)\le \vartheta(G_1\backslash f_1)$ and $\vartheta(G_2)\le \vartheta(G_2\backslash f_2)$. So $\vartheta(G_i\backslash f_i)=\vartheta(G_i)$ for  $i\in [2]$.  Note that  each $\vartheta(G_i)=\vartheta(G_i\backslash v_i)+1$. Therefore, there exist  elements $e_{i_1}$ and $e_{i_2}$  in $M'''$ with $v_1$ and $v_2$ as one of their endpoints respectively. This   contradicts the fact that $v_1=v_2=v$ in $G$ and  $M'''$ is an induced matching of $G$.

(ii) If $|W_i|=n_i$ for all $i\in [2]$, then $M_1\sqcup M_2$   is  an induced matching of $G$. Thus $\vartheta(G)\ge \vartheta(G_1)+\vartheta(G_2)-2$. If $\vartheta(G)\ge \vartheta(G_1)+\vartheta(G_2)-1$ and $M^*=\{e_1,\ldots,e_{\vartheta(G_1)+\vartheta(G_2)-1}\}$ is an induced matching of $G$, then there exist at least $\vartheta(G_1)$ elements in $M^*$ belonging to $E(G_1\backslash f_1)$ or  $\vartheta(G_2)$ elements in $M^*$ belonging to $E(G_2\backslash f_2)$. This forces $\vartheta(G_1)=\vartheta(G_1\backslash f_1)$ or $\vartheta(G_2)=\vartheta(G_2\backslash f_2)$, since every $\vartheta(G_i\backslash f_i)\le \vartheta(G_i)$. Thus there exists an edge $e=\{u,v\}\in M^*$ with $u,v\in V(K_{n_i})$. Since  $|W_i|=n_i$ for all $i\in [2]$, we get that  $u,v$  and the two endpoints of another edge    form a path  with at least three  vertices, a contradiction.
\end{proof}

\subsection{Glue via the $*$ operation}
In this subsection, we study the depth and regularity of the quotient ring of the  edge ideal of a graph, which is obtained  by  $*$ operations. First, we show the first major result of this subsection.

\begin{Theorem}
\label{thm:depth_Fan_Fan_*}
	Suppose that $G=(G_1,f_1)*(G_2,f_2)$,  where $G_i=F_{k_i}^{W_i}(K_{n_i})$ is a $k_i$-fan graph of $K_{n_i}$ on the set $W_i$ with $n_i\ge 2$ and $k_i\ge 1$ for $i\in[2]$.
	 Let $v_i$ be the neighbor of $f_i$ in $G_i$ and $W_i=W_{i,1}\sqcup\cdots\sqcup W_{i,k_i}$ be a partition of $W_i$ with $v_i\in W_{i,1}$ for $i\in[2]$. Let  $T_i=\max\{|W_{i,1}|, |W_{i,2}|,\ldots,|W_{i,k_i}|\}$ and  $T'_i=\max\{|W_{i,1}|-1, |W_{i,2}|,\ldots,|W_{i,k_i}|\}$ for $i\in[2]$, and $t=|\{i:T'_i\ne T_i-1\}|$. Then
	\[
	\depth(S_G/I_{G})=\depth(S_{G_1}/I_{G_1})+\depth(S_{G_2}/I_{G_2})-s
	\]
	where $s=0$ if $t=0$, otherwise, $s=1$.
\end{Theorem}
\begin{proof} First, we have $t \in \{0,1,2\}$ by the definition of $t$. We distinguish between the following two cases:
	
(1) If $t=0$, then $T'_i=T_i-1$ for $i\in[2]$, which forces each $T_i=|W_{i,1}|$ and $|W_{i,1}|>|W_{i,j}|$ for  $j=2,\ldots,k_i$. In this case, we set $J=(x_{N_G(v_2)})+I_{G\backslash N_G[v_2]}$ and $K=(x_{v_2})+I_{G\backslash v_2}$,  then $J+K=(x_{N_G[v_2]})+I_{G \backslash N_G[v_2]}$, $G\backslash v_2=G_1\sqcup (G_2\backslash \{v_2,f_2\})$ and  $G\backslash N_G[v_2]$ has $(|W_2|-|W_{2,1}|+1)$ connected components  consisting of $G_1\backslash f_1$ and $(|W_2|-|W_{2,1}|)$ cliques or isolated vertices. So we have
\begin{align*}
	\depth(S_G/K)=&\depth(S_{G_1}/I_{G_1})+\depth(S_{G_2\backslash \{v_2,f_2\}}/I_{G_2\backslash \{v_2,f_2\}})\\
	=&\depth(S_{G_1}/I_{G_1})+(1+(|W_2|-1)-T'_2)
	\end{align*}
\begin{align*}
=&\depth(S_{G_1}/I_{G_1})+(1+|W_2|-T_2)\\
	=&\depth(S_{G_1 }/I_{G_1})+\depth(S_{G_2 }/I_{G_2}).
\end{align*}
and
\begin{align*}
	\depth(S_G/(J+K))=&\depth(S_{G_1\backslash f_1}/I_{G_1\backslash f_1})+(|W_2|-|W_{2,1}|)\\
	=&\depth(S_{G_1}/I_{G_1})+\depth(S_{G_2}/I_{G_2})-1.
\end{align*}	
where the second equality holds by Lemma \ref{subgraph} and Theorem \ref{depth of F_k}.
Applying Lemma \ref{exact}(\ref{exact-2}) and Lemma \ref{decomposition} to the  exact sequence (\ref{eqn:SES-1}),  we get the desired result.

(2) If  $t\geq 1$, we  assume that $T'_2\neq T_2-1$.   This  implies that
$T_2=|W_{2,j}|\ge |W_{2,1}|$ for some $j\ne 1$. Suppose $T_2=|W_{2,2}|$. In this case, we prove  by induction on $n_2$ that
\[
\depth(S_G/I_{G})=\depth(S_{G_1}/I_{G_1})+\depth(S_{G_2}/I_{G_2})-1.
\]
If $n_2=2$, then  $G_2=P_3\cup_u  K_b$ is the  clique sum of $P_3$ and some $K_b$, we get $\depth(S_{G_2} / I_{G_2})=2$ by Theorem \ref{depth of F_k}. Let $J=(x_{N_G(u)})+I_{G\backslash N_G[u]}$ and $K=(x_{u})+I_{G\backslash u}$,  then $J+K=(x_{N_G[u]})+I_{G \backslash N_G[u]}$, $G\backslash N_G[u]=G_1$ and $G\backslash u= K_{b-1} \sqcup (G_1 \circ P_4)$. So $\depth(S_G/(J+K))=\depth(S_{G_1} / I_{G_1})$ and
\begin{align*}
	\depth(S_G/K)&=\depth(S_{G\backslash u}/I_{G\backslash u})\\
	&=1+\depth(S_{G_1 \circ P_4}/I_{G_1 \circ P_4})\\
	&\geq 1+(\depth(S_{P_4}/I_{P_4})+\depth(S_{G_1}/I_{G_1})-2)\\
	&=\depth(S_{P_4}/I_{P_4})+\depth(S_{G_1}/I_{G_1})-1\\
	&=2+\depth(S_{G_1}/I_{G_1})-1\\
	&=\depth(S_{G_1 }/I_{G_1})+1
\end{align*}
where the  inequality holds by Theorem \ref{thm:depth_Fan_Fan_circ}. Applying Lemma \ref{exact}(\ref{exact-2}) and Lemma \ref{decomposition}  to the  exact sequence (\ref{eqn:SES-1}),  we get
\[
\depth(S_G/I_{G})=\depth(S_{G_1 }/I_{G_1})+1=\depth(S_1/I_{G_1})+\depth(S_2/I_{G_2})-1.
\]

Now we assume that $n_2 \geq 3$.   Let $W_{2,2}=\{w_1,\ldots,w_{r_2}\}$ and $\{K_{a_{2,1}},\ldots,K_{a_{2,r_2}}\}$ be a branch of the fan on $W_{2,2}$ with $V(K_{n_2})\cap V(K_{a_{2,i}})=\{w_1,\ldots,w_i\}$.  Let $J=(x_{N_G(w_1)})+I_{G\backslash N_G[w_1]}$ and $K=(x_{w_1})+I_{G\backslash w_1}$,  then $J+K=(x_{N_G[w_1]})+I_{G \backslash N_G[w_1]}$, $G\backslash N_G[w_1]$ has $(|W_2|-|W_{2,2}|)$  connected components, which  consisting of $G_1$ and $(|W_2|-|W_{2,2}|-1)$ connected components consisting of some cliques and isolated vertices. Then we have 	
\begin{align*}
	\depth(S_G/(J+K))=&\depth(S_{G_1}/I_{G_1})+(|W_2|-|W_{2,2}|-1)\\
	=&\depth(S_{G_1}/I_{G_1})+\depth\,(S_{G_2}/I_{G_2})-2.
\label{eqn:SES-6}
\end{align*}
where the last equality holds	by Theorem \ref{depth of F_k}.   At the same time,   $G\backslash w_1=(G_1 *(G_2\backslash w_1))\sqcup K_{b-1}$, then
\[
\depth(S_G/K)=\depth(S_{G\backslash w_1}/I_{G\backslash w_1})
=1+\depth(S_{G_1 *(G_2\backslash w_1)}/I_{G_1 *(G_2\backslash w_1)}).
\]
To  compute the $\depth(S_{G_1 *(G_2\backslash w_1)}/I_{G_1 *(G_2\backslash w_1)})$,
let $T^*_1=T_1$, $T'^*_1=T'_1$, $T^{*}_2=\max\{|W_{2,1}|,\\ |W_{2,2}|-1,|W_{2,3}|,\ldots,|W_{2,k_2}|\}$,  $T'^*_2=\max\{|W_{2,1}|-1, |W_{2,2}|-1, |W_{2,3}|,
\ldots,|W_{2,k_2}|\}$ and $t'=|\{i:T'^*_i\ne T^*_i-1\}|$.
We distinguish into the following two subcases:

(i) If  $t=1$, then $T'_1 = T_1-1$.  In this case,  for the graph $G_1 *(G_2\backslash w_1)$, we have $t'\leq 1$. Thus, by the above case $(1)$ and by induction, we have
\begin{align*}
 \depth(S_{G\backslash w_1}/I_{G\backslash w_1})&=\begin{cases}
            \depth(S_{G_1}/I_{G_1})+\depth(S_{G_2\backslash w_1}/I_{G_2\backslash w_1}), & \text{if $t'=0$,}\\
            \depth(S_{G_1}/I_{G_1})+\depth(S_{G_2\backslash w_1}/I_{G_2\backslash w_1})-1, & \text{if $t'=1$,}
        \end{cases}\\
 &\ge \depth(S_{G_1}/I_{G_1})+\depth(S_{G_2\backslash w_1}/I_{G_2\backslash w_1})-1.
\end{align*}
 (ii) If $t=2$, then  $T'_1 \neq T_1-1$.  In this case,  for the graph $G_1 *(G_2\backslash w_1)$, we have $t'\geq 1$.  By induction, we have
\[
\depth(S_{G_1 *(G_2\backslash w_1)}/I_{G_1 *(G_2\backslash w_1)})=\depth(S_{G_1}/I_{G_1 })+\depth(S_{G_2\setminus w_1 }/I_{G_2\setminus w_1})-1.
\]
In  both cases we obtain
 \begin{align*}
	\depth(S_G/K)&=1+\depth(S_{G_1 *(G_2\backslash w_1)}/I_{G_1 *(G_2\backslash w_1)})\\
	&\ge 1+(\depth(S_{G_1}/I_{G_1})+\depth(S_{G_2\backslash w_1}/I_{G_2\backslash w_1})-1)\\
	&=\depth(S_{G_1}/I_{G_1})+(1+(|W_2|-1)-T^*_2)\\
	&\ge\depth(S_{G_1}/I_{G_1 })+(|W_2|-T_2)\\
	&=\depth(S_{G_1 }/I_{G_1})+\depth(S_{G_2 }/I_{G_2})-1
\end{align*}
where the second  inequality holds	because of $T^*_2\le T_2$.
Applying Lemma \ref{exact}(\ref{exact-2}) and Lemma \ref{decomposition} to the  exact sequence (\ref{eqn:SES-1}),  we obtain the desired results.
\end{proof}

Next, we are ready to prove another major result of this subsection.
\begin{Theorem}
	\label{thm:reg_Fan_Fan_*}
 Suppose that $G=(G_1,f_1)*(G_2,f_2)$,  where $G_i=F_{k_i}^{W_i}(K_{n_i})$ is a $k_i$-fan graph of $K_{n_i}$ on the set $W_i$ with $n_i\ge 2$ and $k_i\ge 1$ for $i\in[2]$.
	 Let $v_i$ be the neighbor of $f_i$ in $G_i$ and $t=|\{i\in [2]: \reg(S_{G_i\backslash v_i}/I_{G_i\backslash v_i})\ne \reg(S_{G_i}/I_{G_i})\}|$. We have
 \begin{itemize}
 	\item[(1)] If $t\leq 1$, then $\reg(S_G/I_G)=\reg(S_{G_1}/I_{G_1})+\reg(S_{G_2}/I_{G_2})$;
 \item[(2)] If $t=2$ and $|W_i|=n_i-1$ for all $i\in [2]$, then
 \[
 \reg(S_G/I_G)=\reg(S_{G_1}/I_{G_1})+\reg(S_{G_2}/I_{G_2}),
 \]
 \item[(3)]  If $t=2$ and $|W_i|=n_i$ for  some $i\in [2]$, then
 \[
 \reg(S_G/I_G)=\reg(S_{G_1}/I_{G_1})+\reg(S_{G_2}/I_{G_2})-1.
 \]
\end{itemize}		
\end{Theorem}
\begin{proof}
First we have $t\in\{0,1,2\}$  and $\reg(S_G/I_{G})=\vartheta(G)$ and $\reg(S_{G_i}/I_{G_i})=\vartheta(G_i)$ for $i\in [2]$ by the definition of $t$ and Lemma \ref{chordal}.	
We  distinguish between the following two cases:
				
(1) If  $t\leq 1$, we assume that $\reg(S_{G_1\backslash v_1}/I_{G_1\backslash v_1})=\reg(S_{G_1}/I_{G_1})$, i.e., $\vartheta(G_1\backslash v_1)=\vartheta(G_1)$. In this case, $M=M_1 \sqcup M_2$ is an induced matching of $G$, where  $M_1=\{e_{1,1},\ldots,e_{1,\vartheta(G_1\backslash v_1)}\}$  and
$M_2=\{e_{2,1},\ldots,e_{2,\vartheta(G_2)}\}$ are induced matchings  of $G_1\backslash v_1$ and  $G_2$, respectively. Thus $\vartheta(G)\ge \vartheta(G_1\backslash v_1)+\vartheta(G_2)=\vartheta(G_1)+\vartheta(G_2)$.  If $\vartheta(G)\geq \vartheta(G_1)+\vartheta(G_2)+1$, then there exists an induced matching $M'$ of $G$,  which has at least  $\vartheta(G_1)+\vartheta(G_2)+1$ elements. So there exist at least $(\vartheta(G_1)+1)$ elements in $M'$ belonging to $E(G_1)$ or  $(\vartheta(G_2)+1)$ elements in $M'$ belonging to $E(G_2)$. This forces $\vartheta(G_1)+1\le \vartheta(G_1)$ or $\vartheta(G_2)+1\le \vartheta(G_2)$, a contradiction.

(2) If $t=2$, then each $\reg(S_{G_i\backslash v_i}/I_{G_i\backslash v_i})\ne \reg(S_{G_i}/I_{G_i})$, i.e., each $\vartheta(G_i\backslash v_i)<\vartheta(G_i)$. Let  $W_i=W_{i_1}\sqcup\cdots\sqcup W_{i_{k_i}}$ be a partition of $W_i$ with $v_i\in W_{i_1}$. From the assumptions that for each $i$, $\reg(S_{G_i\backslash v_i}/I_{G_i\backslash v_i})\ne\reg(S_{G_i}/I_{G_i})$ and  $f_i$ is a leaf of $G_i$ with a unique neighbor point $v_i$,
we have  $|W_{i_1}|=1$ from the proof of Theorem \ref{thm:fan-reg}.
Let $W_{i_\ell}=\{w_{i_\ell,1},\ldots,w_{i_\ell,r_{i_\ell}}\}$ for $\ell\in[k_i]$ and $i\in [2]$, then by Lemma \ref{clique sum} we get that  $|W_i|\ge  n_i-1$ and   $h_{i_\ell,r_{i_\ell}} \ge 2$ for all $\ell=2,\ldots,k_i$ and $i\in [2]$.
In this case, each $G_i\backslash W_i$ is the disjoint union of $(\vartheta(G_i)-1)$ cliques $K'_{i,1}$, \ldots,$K'_{i,\vartheta(G_i)-1}$ and some isolated vertices. Let $M_i=\{e_{ij} | e_{ij}\in E(K'_{i,j}), j\in[\vartheta(G_i)-1]\}$ for $i\in [2]$, then  $M_i$ is an induced matching  of $G_i\backslash v_i$.  We consider two subcases:
				
(i) If  $|W_i|=n_i-1$ for all $i\in [2]$.  In this case, let $M=M_1\sqcup M_2\sqcup\{e_1,e_2\}$, where the edge  $e_i=\{v_i,u_i\}$ with $u_i\in V(K_{n_i})\backslash W_i$. Then $M$   is  an induced matching of $G$ by the choice of $M_1$ and $M_2$. Hence $\vartheta(G)\ge \vartheta(G_1\backslash v_1)+\vartheta(G_2\backslash v_2)+2= \vartheta(G_1)+\vartheta(G_2)$. If $\vartheta(G)\geq \vartheta(G_1)+\vartheta(G_2)+1$,  then there exists an induced matching $M''$ of $G$,  which has at least  $\vartheta(G_1)+\vartheta(G_2)+1$ elements. So there are at least $(\vartheta(G_1)+1)$ elements in $M''$ belonging to $E(G_1)$ or  $(\vartheta(G_2)+1)$ elements in $M''$ belonging to $E(G_2)$. This forces $\vartheta(G_1)+1\le \vartheta(G_1)$ or $\vartheta(G_2)+1\le \vartheta(G_2)$, a contradiction.

(ii) If $|W_i|=n_i$ for some $i\in [2]$, then we  assume that $|W_1|=n_1$. In this case, let $M=M_1 \sqcup M_2 \sqcup \{e\}$, where the edge $e=\{v_1,f\}$ with $f=f_1=f_2$ in $G$. Then $M$   is  an induced matching of $G$ by the choice of $M_1$ and $M_2$. So $\vartheta(G)\ge \vartheta(G_1\backslash v_1)+\vartheta(G_2\backslash v_2)+1=\vartheta(G_1)+\vartheta(G_2)-1$. If $\vartheta(G)\ge \vartheta(G_1)+\vartheta(G_2)$, then there exists an induced matching $M'''$ of $G$,  which has at least  $\vartheta(G_1)+\vartheta(G_2)$ elements. So there are at least $\vartheta(G_1)$ elements in $M'''$ belonging to $E(G_1)$ or  $\vartheta(G_2)$ elements in $M'''$ belonging to $E(G_2)$.
Since each $\vartheta(G_i)=\vartheta(G_i\backslash v_i)+1$ and $|W_1|=n_1$,
there are  elements $e_{i_1}$ and $e_{i_2}$  in $M'''$ with $e_{i_1}=\{v_1,f\}$ and $e_{i_2}=\{v_2,f\}$. Thus $v_1$, $v_2$ and $f$ form a path in $M'''$, which contradicts that $M'''$ is an induced matching of $G$.				
\end{proof}

\medskip
\hspace{-6mm} {\bf Acknowledgments}

 \vspace{3mm}
\hspace{-6mm}  This research is supported by the Natural Science Foundation of Jiangsu Province (No. BK20221353). The authors are grateful to the computer algebra systems CoCoA \cite{Co} for providing us with a large number of examples.

\medskip
\hspace{-6mm} {\bf Data availability statement}

\vspace{3mm}
\hspace{-6mm}  The data used to support the findings of this study are included within the article.

\medskip

\end{document}